\documentclass[10pt,english]{amsart}

\usepackage{mathrsfs}
\usepackage{amssymb}

\usepackage{amsfonts,amsmath,latexsym,verbatim,amscd,mathrsfs,color,array}
\usepackage{color}
\usepackage{marginnote}
\usepackage{graphicx}

\newtheorem{theorem}   {Theorem}[section]
\newtheorem{proposition}  {Proposition}[section]

\newtheorem{lemma} {Lemma}[section]
\newtheorem{remark}   {Remark}[section]

\newtheorem{notation} {Notation}[section]

\newcommand{\e}{\varepsilon}
\newcommand{\te}{\tilde \varepsilon}
\newcommand{\tS}{\tilde{\mathcal S_n}}
\newcommand{\tN}{\tilde{\mathcal N}}

\newcommand{\Cyl}{\mathbb R \times \left[ \frac{\pi}{2},\frac{3\pi}{2} \right]}

\title{Free boundary minimal surfaces in the unit $3$-ball}

\author{Abigail Folha}
\address{Abigail Folha - Instituto de Matem\'atica - Departamento de Geometria, Universidade Federal Fluminense, R. M\' ario Santos Braga, S/N Campus do Valonguinho 24020-140 Niter\' oi RJ Brazil}
\email{abigailfolha@vm.uff.br}

\author{Frank Pacard}
\address{Centre de MathÃ©matiques Laurent Schwartz, \'Ecole Polytechnique-CNRS}
\email{frank.pacard@math.polytechnique.fr}

\author{Tatiana Zolotareva} 
\address{Centre de MathÃ©matiques Laurent Schwartz, \'Ecole Polytechnique-CNRS}
\email{zolotareva@math.polytechnique.fr}

\begin{document}

\maketitle

\begin{abstract}
In a recent paper A. Fraser and R. Schoen have proved the existence of free boundary minimal surfaces $\Sigma_n$ in $B^3$ which have genus $0$ and $n$ boundary components, for all $ n \geq 3$. For large $n$, we give an independent construction of $\Sigma_n$ and prove the existence of free boundary minimal surfaces $\tilde \Sigma_n$ in $B^3$ which have genus $1$ and $n$ boundary components. As $n$ tends to infinity, the sequence $\Sigma_n$ converges to a double copy of the unit horizontal (open) disk, uniformly on compacts of $B^3$ while the sequence $\tilde \Sigma_n$ converges to a double copy of the unit horizontal (open) punctured disk, uniformly on compacts of $B^3-\{0\}$.
\end{abstract}

\section{Introduction and statement of the result.}

\let\thefootnote\relax\footnotetext{F. Pacard and T. Zolotareva are partially supported by the ANR-2011-IS01-002 grant.}
\let\thefootnote\relax\footnotetext{T. Zolotareva is partially supported by the FMJH through the ANR-10-CAMP-0151-02 grant.}

In this paper, we are interested in minimal surfaces which are embedded in the Euclidean $3$-dimensional unit open ball $B^3$ and which meet $S^2$, the boundary of $B^3$, orthogonally.  Following \cite{Fra-Sch-1}, we refer to such minimal surfaces as {\em free boundary minimal surfaces}. 

\medskip

Obviously, the horizontal unit disk, which is the intersection of the horizontal plane passing through the origin with the unit $3$-ball, is an example of such free boundary minimal surface. Moreover, it is the only free boundary solution of topological disk type, \cite{Nit}. Let $ s_* > 0 $ be the solution of 
\[
s_*\, \tanh s_* =1.
\] 
The so called {\em critical catenoid} parameterized by
\[
(s , \theta) \mapsto \frac{1}{s_* \, \cosh s_*} \, \left( \cosh s \, \cos \theta, \cosh s\, \sin \theta, s\right) ,
\]
is another example of such a {\em free boundary minimal surface}. A. Fraser and M. Li conjectured that it is the only free boundary minimal surface of topological annulus type \cite{Fra-Li}.

\medskip

Free boundary minimal surfaces arise as critical points of the area among surfaces embedded in the unit $3$-ball whose boundaries lie on $S^2$ but are free to vary on $S^2$. The fact that the area is critical for variations of the boundary of the surface which are tangent to $S^2$ translates into the fact that the minimal surface meets $S^2$ orthogonally. 

\medskip

In a recent paper \cite{Fra-Sch-2}, A. Fraser and R. Schoen have proved the existence of free boundary minimal surfaces $\Sigma_n$ in $B^3$ which have genus $0$ and $n$ boundary components, for all $ n \geq 3$. For large $n$, these surfaces can be understood as the connected sum of two nearby parallel horizontal disks joined by $n$ boundary bridges which are close to scaled down copies of half catenoids obtained by diving a catenoid which vertical axis with a plane containing it, which are arranged periodically along the unit horizontal great circle of $S^2$. Furthermore, as $n$ tends to infinity, these free boundary minimal surfaces converge on compact subsets of $B^3$ to the horizontal unit disk taken with multiplicity two. 

\medskip

We give here another independent construction of $\Sigma_n$, for $n$ large enough. Our proof is very different from the proof of A. Fraser and R. Schoen and is more in the spirit of the proof of the existence of minimal surfaces in $S^3$ by doubling the Clifford torus by N. Kapouleas and S.-D. Yang \cite{Kap-Yan}. We also  prove the existence of free boundary minimal surfaces in $B^3$ which have genus $1$ and $n$ boundary components, for all $n$ large enough.

\medskip

To state our result precisely, we define $P_n$ to be the regular polygon with $n$-sides, which is included in the horizontal plane $\mathbb R^2 \times \{0\}$ and whose vertices are given by 
\[
 \left( \cos \left(\frac{2\pi j}{n} \right), \,  \sin \left(\frac{2\pi j}{n} \right), \, 0 \right) \in \mathbb R^3, \qquad \text{for} \qquad j=1, \ldots, n.
\]
We define $\mathfrak S_n \subset O(3)$ to be the subgroup of isometries of $\mathbb R^3$ which is generated by the orthogonal symmetry with respect to the horizontal plane $x_3=0$, the symmetry with respect to coordinate axis $Ox_1$ 
and the rotations around the vertical axis $Ox_3$ which leave $P_n$ globally invariant. 

\medskip

Our main result reads~:
\begin{theorem} \label{MainTheorem}
There exists $n_0 \geq 0$ such that, for each $n \geq n_0$, there exists a genus $0$ free boundary minimal surface $\Sigma_n$ and a genus $1$ free boundary minimal surface $\tilde \Sigma_n$ which are both embedded in $B^3$ and meet $S^2$ orthogonally along $n$ closed curves. 

\medskip

Both surfaces are invariant under the action of the elements of $\mathfrak S_n$ and, as $n$ tends to infinity, the sequence $\Sigma_n$ converges to a double copy of the unit horizontal (open) disk, uniformly on compacts of $B^3$ while the sequence $\tilde \Sigma_n$ converges to a double copy of the unit horizontal (open) punctured disk, uniformly on compacts of $B^3-\{0\}$.
\end{theorem}

Even though we do not have a proof of this fact, it is very likely that (up to the action of an isometry of $\mathbb R^3$), the surfaces $\Sigma_n$ coincide with the surfaces already constructed by R. Schoen and A. Fraser. In contrast, the existence of  $\tilde \Sigma_n$ is new and does not follow from the results in \cite{Fra-Sch-2}. The parameterization of the free boundary minimal surfaces we construct is not explicit, nevertheless our construction being based on small perturbations of explicitly designed surfaces, it has the advantage to give a rather precise description of the surfaces $\Sigma_n$ and $\tilde \Sigma_n$. Naturally, the main drawback is that the existence of the free boundary minimal surfaces is only guaranteed when $n$, the number of boundary curves, is large enough. 

\section{Plan of the paper.}

In section 3, we study the mean curvature of surfaces embedded in $B^3$ which are graphs over the horizontal disk $ D^2 \times \{0\} $. In section 4 we analyse harmonic functions which are defined on the unit punctured disk in the Euclidean $2$-plane and have log type singularities at the punctures. In section 5 for every $n \in \mathbb N$ large enough we construct a family of genus 0 surfaces $\mathcal S_n$ and a family of genus 1 surfaces $\tS$ embedded in $B^3$ which are approximate solutions to the minimal surface equation, meet the unit sphere $S^2$ orthogonally and have $n$ boundary components. In section 6 we consider all embedded surfaces in $B^3$ which are close to $\mathcal S_n$ and $\tS$ and meet the sphere $S^2$ orthogonally. In section 7 we analyse the linearised mean curvature operator about $\mathcal S_n$ and $\tS$. Finally, in the last section we explain the Fixed-Point Theorem argument that allows us for $n$ large enough to deform $\mathcal S_n$ and $\tS$ into free boundary minimal surfaces 
$\Sigma_n$ and $\tilde \Sigma_n$ satisfying the theorem \eqref{MainTheorem}.

\section{The mean curvature operator for graphs in the unit $3$-ball}

We are interested in surfaces embedded in $B^3$ which are graphs over the horizontal disk $D^2 \times \{0\}$. To define these precisely, we introduce the following parametrization of the unit ball
\[
X (\psi,\phi, x_3) : =  \frac{1}{\cosh x_3 + \cos \psi} \, \left(  \sin \psi \, e^{i \phi} , \sinh x_3 \right),
\]
where $ \psi \in (0, \pi/2)$, $ \phi \in S^1 $ and $ x_3 \in \mathbb R $. The horizontal disk $D^2 \times \{0\}$ corresponds to $ x_3 = 0 $ in this parametrization and the unit sphere $S^2$ corresponds to $ \psi =\pi/2 $. Also, the leaf $x_3 = x_3^0$ is a constant mean curvature surface (in fact it is a spherical cap) with mean curvature given by
\[
H = 2 \, \sinh x_3^0 ,
\]
(we agree that the mean curvature is the {\em sum} of the principal curvatures, not the average) moreover, this leaf  meets $S^2$ orthogonally. 

\medskip

In these coordinates, the expression of the Euclidean metric is given by
\[
X^*g_{eucl} =  \frac{1}{(\cosh x_3 + \cos \psi)^2} \, \left(  d\psi^2 + (\sin \psi)^2 \, d\phi^2 + dx_3^2 \right).
\]

We consider the coordinate  
\[
z = \frac{\sin \psi}{1+ \cos \psi} \, e^{i\phi} ,
\]
which belongs to the unit disk $D^2 \subset \mathbb C$. We then define $ \mathcal X$ by the identity
\[
\mathcal X \left( z , x_3 \right) =  X (\psi, \phi, x_3),
\]
where $z$ and $(\psi,\phi)$ are related as above. Then
$$ 
\mathcal X(z,x_3) = A(z,x_3)(z, B(z) \, \sinh x_3) ,
$$
where the functions $B$ and $A$ and explicitly given by
\[
B(z) = \frac{1}{2} \left( 1 + |z|^2 \right), \quad A (z,x_3) := \frac{1}{1 + B(z)(\cosh x_3 - 1)}.
\]
In the coordinates $ z \in D^2 $ and $ x_3 \in \mathbb R $ the expression of the Euclidean metric is given by
\[
\mathcal X^* g_{eucl} =  A^2(z,x_3) \, \left(  dz^2 +  B^2(z) \, dx_3^2 \right).
\]

\medskip

In the next result, we compute the expression of the mean curvature of the graph of a function $z \mapsto u(z)$ in $B^3$, and by such a graph we mean a surface parametrized by 
$$ 
z \in D^2 \, \mapsto \, \mathcal X(z,u(z)) \in B^3. 
$$
We have the:
\begin{lemma}
\label{meancurvatureGraph}
The mean curvature with respect to the metric $\mathcal X^* g_{eucl}$ of the graph of the function $u$, namely the surface parametrized by $(z, u(z))$, is given by
\[
H(u) =  \frac{1}{A^3(u) \, B} \, \mbox{\rm div} \, \left( \frac{A^2(u) \, B^2 \, \nabla u}{\sqrt{1+ B^2\, |\nabla u|^2}} \right) + 2 \, \sqrt{1+ B^2\, |\nabla u|^2} \,  \sinh u ,
\]
where by definition $A (u)= A (\cdot , u )$. In this expression, the metric used to compute the gradient of u, the divergence and the norm of $\nabla u$ is the Euclidean metric on $D^2$. 
\end{lemma}
\begin{proof}
The area form of the surface parametrized by $z = x_1 + i \, x_2 \mapsto (z, u(z))$ is given by
\[
da : =  A^2(u) \, \sqrt{1+ B^2\, |\nabla u|^2} \, dx_1 \, dx_2 ,
\]
and hence the area functional is given by
\[
\text{\rm Area} (u) : =  \iint_{D^2} A^2(u) \, \sqrt{1 + B^2\, | \nabla u |^2} \, dx_1 \, dx_2 .
\]
The differential of the area functional at $u$ is given by
\[
\left. D\text{\rm Area} \right|_{u} (v) =  \iint_{D^2} \left( - \frac{ A^2(u) \, B^2 \, \nabla u \cdot \nabla v }{ \sqrt{1+ B^2\, |\nabla u|^2} }  + 2 A(u) \, \partial_{x_3} A (u)\sqrt{ 1+ B^2 |\nabla u|^2 } \, v \right)  \, dx_1 \, dx_2 .
\] 
But
\[
\partial_{x_3} A =  - A^2 \, B \, \sinh x_3,
\]
and hence we conclude that
\[
\begin{array}{lllll}
\left. D \text{\rm Area} \right|_{u} (v) = \\[3mm]
\displaystyle   - \iint_{D^2} \left(  \text{\rm div} \left( \frac{A^2(u) \, B^2 \nabla u}{\sqrt{1+ B^2\, |\nabla u|^2}} \right) + 2 \, A^3 (u) \, B \, \sqrt{1+ B^2\, |\nabla u|^2} \, \sinh u \right) \, v \, dx_1 \, dx_2 .
\end{array}
\] 

To conclude, observe that the unit normal vector to the surface parametrized by $z \mapsto \mathcal X(z, u(z))$ is given by
\[
N : =  \frac{1}{A(u)} \, \frac{1}{\sqrt{1+B^2\, |\nabla u|^2}} \left(  - B \, \nabla u + \frac{1}{B} \, \partial_{x_3} \right),  
\]
and hence
\[
g_{eucl} (N, \partial_{x_3}) = \frac{A(u) \, B}{\sqrt{1 + B^2\, |\nabla u|^2}},
\]
so that 
\[
g_{eucl} (N, \partial_{x_3}) \, da = A^3(u) \, B, 
\]
and the result follows from the first variation of the area formula
\[
D\text{\rm Area}_{|u} (v) =  - \iint_{D^2}  H(u) \, g_{eucl} (N, \partial_{x_3}) \, v \, da.
\]
This completes the proof of the result.
\end{proof}

\medskip

Using the above Lemma, we obtain the expression of the linearised mean curvature operator about $u=0$. It reads 
\begin{equation}
L_{gr} \, v = \Delta(B v) = \Delta \left(\frac{1+|z|^{2}}{2} \, v  \right) ,
\label{eq:3.2}
\end{equation}
where $\Delta$ is the (flat) Laplacian on $D^2$.

\begin{lemma}
 Take a change of variables in $ D^2 :  z = r \, e^{i \phi}, \ r \in (0,1), \, \phi \in S^1 $ and a function $ u \in \mathcal C^1(D^2) $, such that $ 
 \left. \frac{\partial u}{\partial r} \right|_{r = 1} = 0 $. Then the graph of $u$ in $B^3$ meets the sphere $ S^2 $ orthogonally at the boundary.
\end{lemma}

\begin{proof}
 The surface parametrized by $ (r,\phi) \, \mapsto \, \mathcal X(r \, e^{i \phi}, u(r,\phi)) $ is embedded in $B^3$ and meets $\partial B^3$ at $r 
 = 1 $. The result follows from the fact that the tangent vector 
 $$ 
 T_r (1) = \left. \partial_r \, \mathcal X(r \, e^{i \phi}, u(r,\phi)) \right|_{r=1} = \dfrac{1}{\cosh^2 u(1,\phi)} \left( e^{i\phi}, 
 \, \sinh u(1,\phi) \right) ,
 $$
 is collinear to the normal vector 
 $$ 
 N_S = \dfrac{1}{\cosh u(1,\phi)} \left( e^{i \phi}, \sinh u(1,\phi) \right) ,
 $$ 
 to the sphere $S^2$ at the point $ \mathcal X(e^{i \phi}, u(1,\phi)) $.
\end{proof}

\section{Harmonic functions with singularities defined on the unit disk}

Take some number $ n \in \mathbb N $. Our goal is to construct a graph in $B^3$ which has bounded mean curvature, is invariant under the transformation $z \mapsto \bar z$ and the rotations by $\frac{2\pi}{n}$ and is close to a half-catenoid in small neighbourhoods of the $n$-th roots of unity 
$$ 
z_m = e^{\frac{2\pi i m}{n}} \in \partial D^2, \ m = 1, \ldots n ,
$$
(and, in the case of the second construction, to a catenoid in a small neighbourhood of $z=0$). 

\smallskip

The parametrisation of a standard catenoid $C$ in $\mathbb R^3$ is
$$ 
X^{cat}(s, \theta) = \left( \cosh s \, e^{i \phi}, s \right), \quad (s,\phi) \in \mathbb R \times S^1. 
$$ 
It may be divided into two pieces $C^\pm$, which can be parametrized by
$$ 
z \in \mathbb C \setminus D^2 \, \mapsto \left( z, \, \pm \log |z| \mp \log 2 + \mathcal O(|z|^{-2}) \right), \quad \mbox{as} \quad |z| \rightarrow \infty .
$$

We would like to find a function $ \Gamma_n $, which satisfies 
\begin{equation}  \label{Green1}
 \left\{ \begin{array}{l} L_{gr} \, \Gamma_n = 0 \quad \text{in} \quad D^2 \ (D^2 \setminus \{0\}) \\[3mm] 
 \partial_r \Gamma_n = 0 \quad \text{on} \quad \partial D^2 \setminus \{ z_1, \ldots, z_n \} \end{array} \right. ,
\end{equation}
and which has logarithmic singularities at $z=z_m$ (and $z=0$). Notice that the operator $L_{gr}$ in the unit disk with Neumann boundary data has a kernel which consists of the coordinate functions $x_1, x_2$. This corresponds to tilting the unit disk $ D^2 \times \{0\} $ in $B^3$. The kernel can be eliminated by asking $\Gamma_n$ to be invariant under the action of a group of rotations around the vertical axis.

\smallskip

Notice also that the constant functions are not in the kernel of $L_{gr}$~: by moving the disk in the vertical direction in the cylinder $ D^2 \times \mathbb R $ we do not get a minimal but a constant mean curvature surface in $B^3$. 

\medskip

Take a function $G_n$, such that $ G_n(z^n) = B(z) \, \Gamma_n(z) $. Then the problem \eqref{Green1} is equivalent to
\begin{equation} \label{Green2} 
 \left\{ \begin{array}{l} \Delta G_n = 0  
 \quad \text{in} \quad D^2 \ (D^2 \setminus\{ 0 \}) \\[3mm] 
 \partial_r G_n - \frac{1}{n} G_n =  0 \quad \text{on} \ \partial D^2 \setminus \{1\} .
 \end{array} \right.
\end{equation}

\smallskip

We construct $G_n$ explicitly. For all integer $ n \geq 2 $, we put
\begin{equation}
G_n (z) :=  -\frac{n}{2} + \mathrm{Re} \left( \sum_{j=1}^\infty \frac{n \, z^j}{n j - 1} \right).  
\label{eq:3.2}
\end{equation}
Writing
\[
\frac{1}{n j - 1} =  \sum_{k=0}^\infty \frac{1}{(nj)^{k+1}},
\]
we see that also have the expression
\begin{equation}
G_n (z) : =  -\frac{n}{2} +  \mathrm{Re} \left( \sum_{k=0}^\infty \frac{H_k(z)}{n^k}  \right) ,
\label{eq:3.3}
\end{equation}
where, for all $k\in {\bf N}$, the function $H_k$ is given by
\begin{equation}
H_k (z) : = \sum_{j=1}^\infty \frac{z^j}{j^{k+1}}.
\label{eq:3.4}
\end{equation}
Observe, and this will be useful, that 
\begin{equation}
H_0(z) = - \ln (1-z).
\label{eq:3.5}
\end{equation}

Obviously, $G_n$ is harmonic in the open unit disk. Making use of (\ref{eq:3.4}), we see that, for all $k\geq 1$,  
\[
\partial_r \left( \mathrm{Re} \, H_k \right) = \mathrm{Re} \, H_{k-1},
\]
on $\partial D^2$, while it follows from (\ref{eq:3.5}) that 
\[
\partial_r \left(  \mathrm{Re} \, H_0 \right) = \frac{1}{2} ,
\]
again on $\partial D^2-\{1\}$. Therefore, we conclude from (\ref{eq:3.3}) that 
\[
n\, \partial_r G_n - G_n =0 ,
\]
on $\partial D^2$. 

\medskip

For all integer $n\geq 1$, we define in $D^2 -\{0\}$, the function $\tilde G_n$ by
\begin{equation}
\tilde G_n (z ) : =  - n - \log |z| .
\label{eq:3.1}
\end{equation}
Again $\tilde G_n$ is harmonic in $D^2-\{0\}$ and we also have
\[
n\, \partial_r \tilde G_n - \tilde G_n =0,
\]
on $\partial D^2$. 

\medskip

To complete this paragraph, we define 
\[
\Gamma_n (z) : =  \frac{1}{B(z)} \, G_n(z^n) \qquad \text{and} \qquad \tilde \Gamma_n (z) : =  \frac{1}{B(z)} \, \tilde G_n(z^n).
\]
By construction, $ L_{gr} \, \Gamma_n =  0 $ in $D^2$ and  $ \partial_r \Gamma_n =0 $ on $ \partial D^2 $, away from the $n$-th roots of unity ; while $ L_{gr} \, \tilde \Gamma_n  =0 $ in $ D^2-\{0\} $ and $ \partial_r \tilde \Gamma_n =0 $ on $ \partial D^2 $.

\subsection{Matching Green's function}

Take two parameters $ 0 < \e < 1 $ and $0 < \te < 1$ and consider a catenoid $C_{\te}$ in $\mathbb R^3$, parametrized by
$$ 
X^{cat}_{\te}: \, (s,\phi) \in \mathbb R \times S^1 \, \mapsto \, \left( \te \cosh s \, e^{i \phi}, \te s \right) ,
$$
and $n$ half-catenoids $C_{\e,m}$, $m=1,\ldots,n$
$$ 
X^{cat}_{\e,m} : \, (s,\theta) \in \mathbb R \times \left[ \frac{\pi}{2}, \frac{3\pi}{2} \right] \mapsto \left( \e \cosh s \, e^{i \theta} + z_m, \e s \right) , 
$$
centered at the $n$-th roots of unity $z_m$. In a neighbourhood of $ z = 0 $ ($ z=z_m $), we can take a change of variables
\begin{multline*} 
 z = \te \cosh s \, e^{i\phi}, \ \ ( z = z_m + \e \cosh s \, e^{i \theta} ), \quad \phi \in S^1, \ \ ( \theta \in \left[ -\theta_\e + 2\pi m /n, \theta_\e + 2 \pi m/n \right] ) \\[3mm] 
 s \in [-s_{\te}, 0] \quad \mbox{and} \quad s \in [0,s_{\te}], \ \ ( s \in [-s_{\e}, 0] \quad \mbox{and} \quad s \in [0,s_{\e}] ),
\end{multline*} 
for certain parameters $s_{\te}$, $s_\e \in (0, + \infty)$ and $\theta_\e \in (0,\pi/2)$. We can parametrize the lower and the upper parts of $C_{\te}$ and $C_{\e,m}$ as graphs
$$ 
z \, \mapsto \, \left( z, \pm \, G_{\te}^{cat} \right), \quad z \, \mapsto \, \left( z, \pm \, G_{\e,m}^{cat} \right) ,
$$
where in some small neighbourhoods of $z = 0$ ($z=z_m$),
\begin{multline*}
 G_{\te}^{cat}(z) = \te \log \frac{\te}{2} - \te \log |z| + \mathcal O \left( {\te}^3 / |z|^2 \right), \\  
 G_{\e,m}^{cat}(z) =  \e \log \frac{\e}{2} - \e \log |z - z_m| +  \mathcal O \left( \e^3 / |z-z_m|^2 \right) 
\end{multline*}
Our goal is to find positive parameters $\tau$ and $\tilde \tau$ and a connection between $\e$ and $\te$, such that the function 
$$ 
z \, \mapsto \, \tau \, G_n(z^n) \, + \, \tilde \tau \, \tilde G_n(z^n) ,
$$ 
would be close to $G_{\te}^{cat}$ in a neighbourhood of $z=0$ and to $G_{\e,m}^{cat}$ in a neighbourhood of $z=z_m$. We denote
$$ 
f_n(z):=\sum_{k=0}^\infty \frac{H_k(z^n)}{n^k} = \sum_{k=0}^\infty \frac{1}{n^k} \sum_{j=1}^\infty \frac{z^{nj}}{j^{k+1}} ,
$$
and remind that $ G_n(z^n) = - \frac{n}{2} + \mathrm{Re} \, f_n(z) $. It is easy to verify the function $f_n(z)$ satisfies
$$ 
\dfrac{\partial f_n}{\partial z}(z) = - \frac{d}{dz} \log(1-z^n) + \frac{1}{z} f_n(z) ,
$$
which yields
$$ 
\dfrac{d}{dz} \left( \dfrac{f_n}{z} \right) = \dfrac{nz^{n-2}}{z^n - 1} .
$$
We can write
\begin{multline*}
 \dfrac{ nz^{n-2}}{z^n-1} -  \dfrac{1}{z_m(z-z_m)} = \dfrac{ n z^{n-2} z_m - \sum\limits_{k=0}^{n-1} z^{n-1-k} z_m^k }{ z_m(z^n - 1) } = \\[3mm] 
 \dfrac{ \sum\limits_{k=0}^{n-1} \left( - z^{n-1-k} z_m^k + z_m z^{n-2} \right) }{ z_m (z^n - 1) } = 
 \dfrac{ z^{n-2}( z_m - z ) + z_m \sum\limits_{k=2}^{n-1} z^{n-1-k} \left( z^{k-1} - z_m^{k-1} \right) }{ z_m(z-z_m) \sum\limits_{k=0}^{n-1} z^{n-1-k}z_m^k }  
 = \\[3mm]
 \dfrac{ - z^{n-2} + z_m \sum\limits_{k=2}^{n-1} z^{n-1-k} \sum\limits_{l=0}^{k-2} z^{k-2-l} z_m^l }{ z_m \sum\limits_{k=0}^{n-1} z^{n-1-k}z_m^k } := h_n(z).
\end{multline*}
The function $h_n(z)$ is continuous in a small neighbourhood of $ z = z_m $ and 
$$ 
| h_n(z_m) | \leq c \, n ,
$$
for a constant $c$ which does not depend on $n$. So, we have
$$ 
\dfrac{d}{dz} \left( \dfrac{f_n}{z} \right) + \frac{1}{z_m(z-z_m)} = h_n(z) ,
$$
which yields that in a neighbourhood of $ z = z_m $
$$ 
\dfrac{f_n(z)}{z} + \dfrac{1}{z_m} \, \log( z - z_m ) = \dfrac{1}{z_m} \, \underset{z \rightarrow z_m}{\lim} \left( f_n(z) + \log(z-z_m) \right) + \int_{z}^{z_m} h_n(z)dz ,
$$
where the integral is taken along the segment of the straight line passing from $z$ to $z_m$ and by $\log$ we mean the principal value of complex logarithm defined in the unit disc deprived of a segment of a straight line which doesn't pass thought any of the $n$-th roots of unity. We have
$$ 
\sum_{k=1}^\infty \dfrac{1}{n^k} H_k(z_m^n) = \sum_{k=1}^\infty \sum_{j=1}^\infty \frac{1}{ n^k \, j^{k+1}} \leq \dfrac{\pi^2}{6n} .
$$
Moreover,
$$ 
\mathrm{Re} \underset{ z \rightarrow z_m }{\lim} \left( - \log(1 - z^n) + \log(z-z_m) \right) = - \log | n \, z_m^{n-1} | = - \log n .
$$

So, in the neighbourhood of $ z = z_m $, we have
$$ 
G_n(z^n) = - \frac{n}{2} + c(n) + \log | z-z_m | + \mathcal O(|z-z_m| \log |z-z_m|) + \mathcal O(n|z-z_m|) , 
$$
where $ |c(n)| \leq c \, \log n $ for a constant $c$ which does not depend on $n$ and
$$ 
\tau \, G_n(z^n) \, + \, \tilde \tau \, \tilde G_n(z^n) = \left\{ \begin{array}{l} - n \left( \tilde \tau + \tau/2 \right) - \tilde \tau n \log |z| + \mathcal O(\tau |z|^n), \ \text{as} \ |z| \rightarrow 0 \\[3mm]
-n \left( \tilde \tau + \tau/2 \right) + \tau \, c(n) - \tau \log |z - z_m| \, + \\[3mm]
\mathcal O(\tau |z-z_m| \log |z-z_m|) + \mathcal O( \tau n|z-z_m| ), \ \text{as} \ |z-z_m| \rightarrow 0 \end{array} \right. .
$$
\smallskip
We should take $ \tau  = \e $ and $ n \tilde \tau  = \te $. Moreover, we should have
$$ 
- \te - \dfrac{\e n}{2} = \te \log \dfrac{\te}{2} \quad \mbox{and} \quad - \te - \dfrac{\e n}{2} + \e \, c(n) = \e \log \dfrac{\e}{2} .
$$
This gives us the relation
$$ 
\log \dfrac{\e}{\te} + \dfrac{\te}{\e} - \dfrac{n}{2} \, \dfrac{\e}{\te} = - \dfrac{n}{2} + c(n) + 1 ,
$$
and 
$$ 
\dfrac{\e}{\te} = g_n^{-1}( - \dfrac{n}{2} + c(n) + 1 ) = : d(n) ,
$$
where 
$$ 
g_n(t) \, : \, t \in (0, + \infty) \, \mapsto \, \log t - \dfrac{n}{2} \, t + \dfrac{1}{t} \in (- \infty, \infty) ,
$$
is an everywhere decreasing function. Finally, we find
$$ 
\te = 2 \, e^{ -1 - \frac{n}{2} \, d(n) } \quad \mbox{and} \quad \e = d(n) \, \te .
$$
(In the case, where we do not have a singularity at $z=0$ we just need to take $ \e = e^{ - \frac{n}{2} + c(n)} $). We obtain for all $\beta \in (0,1)$
\begin{multline} \label{functionG}
 \tau \, G_n(z^n) = \e \log \dfrac{ \e }{ 2|z-z_m| } + \mathcal O(\e^{1 - \beta} \, |z-z_m|), \ \text{as} \ |z-z_m| \rightarrow 0 \\[3mm]
 \tau \, G_n(z^n) \, + \, \tilde \tau \, \tilde G_n(z) = \left\{ \begin{array}{l} \te \log \dfrac{\te}{2|z|} + \mathcal O( \te^{1-\beta} \, |z|^n), \ 
 \text{as} \ |z| \rightarrow 0 \\[3mm]
 \e \log \dfrac{\e}{2|z-z_m|} + \mathcal O(\e^{1-\beta} \, |z-z_m|), \quad \text{as} \ |z-z_m| \rightarrow 0 \end{array} \right.
\end{multline}
Finally, we put $ \, \mathcal G_n(z) = \tau \, G_n(z^n) / B \, $ and $ \, \tilde { \mathcal G }_n(z) = \left( \tilde \tau \, \tilde G_n(z^n) + \tau \, G_n(z^n) \right) / B \, $ and
$$ 
\tilde{\mathcal G}_n(z) = \left\{ \begin{array}{l} 2 \, \te \log \dfrac{\te}{2|z|} + \mathcal O( \te^{1-\beta} \, |z|^2), \ \text{as} \ |z| \rightarrow 0 \\[3mm]
\e \log \dfrac{\e}{2|z-z_m|} + \mathcal O(\e^{1-\beta} \, |z-z_m|), \ \text{as} \ |z-z_m| \rightarrow 0 \end{array} \right. 
$$

\begin{remark}
We can now explain why our construction works only for large $n$. On one hand, in order to match the graph of the Green's function $ \tilde{\mathcal G}_n$ with catenoids we need to truncate the cantenoids far enough and scale them by a small enough factor. On the other hand, in the neighbourhood of singularities the constant term of $ \tilde{\mathcal G}_n$ depends on the number of singularities $n$ and, as constant functions are not in the kernel of the linearised mean curvature operator, this gives the correspondence between the scaling factors of the catenoids and $n$.
\end{remark}

\section{Catenoidal bridges and necks}

In this section we explain the construction of the surface $\mathcal S_n$, invariant under the action of the group $\mathfrak S_n$, which has bounded mean curvature, meets the unit sphere $S^2$ orthogonally at the boundary and is close to two horizontal disks "glued together" with the help of "catenoidal bridges" in the neighbourhood the $n$-th roots of unity. We will also denote $\tilde {\mathcal S_n}$ the genus 1 surface, obtained from $\mathcal S_n$ by attaching a "catenoidal neck" at $z=0$. We will explain now what we mean by "catenoidal bridges" and "catenoidal neck".

\subsection{Catenoidal bridges}

One of the possible constructions would be to "glue" the graph $ \mathcal X(z, \mathcal G_n(z)) $ together with minimal stripes obtained by intersecting euclidean catenoids centered at $ z=z_m $ with the unit sphere. The difficulty of this approach is that those stripes would not meet the sphere orthogonally. We prefer to find a way to put half-catenoids in the unit sphere in the orthogonal way loosing the minimality condition. 

\medskip

\begin{remark}
We describe below the construction of the surface $\tS$. The construction of the surface $ \mathcal S_n $ can be obtained by replacing the function $\tilde{\mathcal G}_n$ by $ \mathcal G_n $ which has the same expansion in the terms of $\e$ at $z = z_m$, taking into account that the relation between the parameters $n$ and $\e$ changes. 
\end{remark}

We use the notation $ \mathbb C_- $ for the half-plane $ \{ z \in \mathbb C \, | \, \mathrm{Re}(z) < 0 \} $. For $ m = 1, \ldots, n $ consider the conformal mappings 
$$
\lambda_m : \mathbb C_- \longrightarrow D^2, \quad \lambda_m(\zeta) = e^{\frac{2 i \pi m}{n}} \frac{1+\zeta}{1-\zeta}. 
$$
These mappings transform a half-disk in the $ \mathbb C_- $ centered at $ \zeta = 0 $ and of radius $\rho < 1$ to a domain obtained by the intersection of the unit disk $D^2$ with a disk of radius $\frac{2\rho}{1-\rho^2}$ and a center at $\frac{1+\rho^2}{1-\rho^2} e^{\frac{2 i \pi m}{n}}$. Let $ ( \zeta = \xi_1 + i \xi_2, \xi_3) $ be the coordinates in $\mathbb C_- \times \mathbb R$, then we define the mapping 
$$ 
\Lambda_m : \mathbb C_- \times \mathbb R \longrightarrow D^2 \times \mathbb R, \quad \Lambda_m(\zeta, \xi_3) = \left( \lambda_m(\zeta), 2 \, \xi_3 \right). 
$$

\medskip

Consider the half-catenoid $C_{\e/2}$ in $ \mathbb C_- \times \mathbb R $, parametrized by
$$ 
X^{cat}_{\e/2} : \, (\sigma, \theta) \in \mathbb R \times \left[ \frac{\pi}{2}, \frac{3\pi}{2} \right]  \, \mapsto \, \left( \frac{\e}{2} \cosh \sigma \, e^{i\theta}, \frac{\e}{2} \sigma \right) $$
In the regions, where $ \sigma > 0 $ or $ \sigma < 0 $ we can take the change of variables 
$$ \zeta =  \frac{\rho}{2} \, e^{i \theta} = \frac{\e}{2} \cosh \sigma \, e^{i \theta}, \quad \theta \in \left[ \frac{\pi}{2}, \frac{3\pi}{2} \right] $$ 
and, having in mind that the function $ \tilde{\mathcal G}_n $ defined in $\bar D^2 \setminus \{ z_1, \ldots, z_n \}$ is invariant under rotations by the angle $\frac{2\pi}{n}$, consider a vertical graph over $\mathbb C_-$:
$$ (\rho, \theta) \, \mapsto \, \left( \frac{\rho}{2} \, e^{i \theta}, \frac{1}{2} \, \bar{ \mathcal G}_n (\rho,\theta) \right), \ \text{where} \ \ \bar{ \mathcal G }_n(\rho,\theta)  = \tilde{\mathcal G}_n(\lambda_m(\rho/2 \, e^{i \theta})) $$

In the neighbourhood of $ \rho = 0 $, we have $ | \lambda_m(\rho/2 \, e^{i \theta}) - z_m | = \rho + \mathcal O(\rho^2) $. So, using the expansion \eqref{functionG} for the function $ \tilde{ \mathcal G}_n $ in the neighbourhood of $z = z_m$, we obtain a similar expansion for $ \bar {\mathcal G_n} $ in the neighbourhood of $0$:
$$ 
\bar{ \mathcal G }_n(\rho,\theta) = \e \log \frac{\e}{2 \rho} + \mathcal O(\e^{1-\beta} \rho ), \quad \forall \beta \in (0,1).  
$$
At the same time the lower and the upper part of the $C_{\e/2}$ can be seen as graphs of the functions 
$$ 
\pm \, G^{cat}_{\e/2}(\rho) = \pm \, \frac{\e}{2} \log \frac{\e}{2 \rho} + \mathcal O(\e^3/\rho^2). 
$$ 
Now take a function $ \bar \Upsilon $ which is defined in the neighbourhood of $ |\zeta| = \e $ by
$$ 
\bar \Upsilon(\rho,\theta) = (1 - \bar  \eta_\e(\rho)) \, \frac{1}{2} \, \bar{ \mathcal G }_n(\rho,\theta)  + \bar \eta_\e(\rho) \, G_{\e/2}^{cat}(\rho) ,
$$
where $\bar \eta_\e$ is a cut-off function, such that
$$ 
\bar \eta_\e \equiv 1, \quad \mbox{for} \quad \e < \rho < 1/2 \, \e^{2/3}, \quad \bar \eta_\e \equiv 0, \quad \mbox{for} \quad  \rho > \e^{2/3} .
$$

Using this, we can parametrize the surface $ \mathcal S_n $ in the region 
$$ 
\Omega_{cat}^m: = \lambda_m \left\{ \e / 2 \, \cosh \sigma \, e^{i\theta} \, : \, \e  \cosh \sigma < 1/2 \, \e^{2/3}, \ \theta \in \left[ \pi/2, 3\pi/2 \right]  \right\} ,
$$ 
by $ (\sigma,\theta) \mapsto \, \mathcal X \circ \Lambda_m \left( \e / 2 \, \cosh \sigma \, e^{i\theta}, \frac{\e}{2} \sigma \right) $ 
and as a bi-graph
\begin{multline*} 
 \mathcal X \left\{ ( z, 2 \,\Upsilon^m(z) ) \cup (z, - 2 \, \Upsilon^m(z)) \right\}, \\
 \text{for} \ z \in \Omega_{glu}^m : = \lambda_m \left\{ \rho /2 \, e^{i \theta} \, : \, 1/2 \, \e^{2/3} < \rho < 2 \, \e^{2/3}, \ \theta \in \left[  \pi/2, 3\pi/2 \right] \right\}  ,
\end{multline*}
where $\Upsilon^m(z)$ is a function, such that $ \Upsilon^m(\lambda_m(\rho/2 \, e^{i\theta})) = \bar \Upsilon(\rho,\theta) $.

\medskip

\begin{remark}{\textbf{Orthogonality at the boundary}}
\end{remark}
In a neighbourhood of its $m$-th component of the boundary the surface $\tS$ ($\mathcal S_n$) can be seen as the image by the mapping $ \mathcal X \circ \Lambda_m$ of a surface (which we denote $ \bar S_n$) contained in the half-space $\mathbb C_- \times \mathbb R$. Consider a foliation of the half-space by horizontal half-planes. It is clear that every leaf of this foliation is orthogonal to $\partial \mathbb C_- \times \mathbb R$. Thus, the normal to $\partial \mathbb C_- \times \mathbb R$ at a point is tangent to the horizontal leaf passing through this point. So, if there existed a tangent vector field along $ \bar S_n$, horizontal at $\partial \bar S_n$, then it would have to be collinear to the normal to $\partial \mathbb C_- \times \mathbb R$. 

\medskip

On the other hand, the image of the the foliation by horizontal half-planes by the mapping $\mathcal X \circ \Lambda_m$ gives a foliation of the unit ball by spherical caps which are orthogonal to $S^2$ at the boundary. The horizontal vector field tangent to $ \partial \mathbb C_- \times \mathbb R $ is sent by this mapping to a vector field tangent to the sphere and to a spherical cap leaf. The result follows from the fact that the restriction of $\mathcal X \circ \Lambda_m$ to horizontal half-planes is conformal.

\medskip

Finally in our case, the existence of the horizontal tangent vector field follows from the fact that $\partial_\theta X^{cat}_{\e/2}$ is horizontal and that
$ \partial_\theta \bar{\mathcal G}_n = \partial_\theta G_{\e/2}^{cat} = \partial_\theta \bar \eta_\e = 0$ at $\theta \in \{ \pi/2, 3 \pi /2 \}$.

\medskip

Let $\mathcal H$ denote the mean curvature of the surface $\tS$.

\begin{proposition} \label{cat0}
 There exists a constant $c$ which does not depend on $\e$ such that in the region 
 $$ 
 \Omega_{cat}^m = \lambda_m \left\{ \e/2 \, \cosh \sigma \, e^{i\theta} \, : \, \e \, \cosh \sigma < 1/2 \, \e^{2/3}, \ \theta \in \left[ \pi/2, 3\pi/2 \right] 
 \right\}, 
 $$ 
 we have
 $$ 
 \left| \mathcal H( \lambda_m ( \e/2 \, \cosh \sigma \, e^{ i \theta}) \right| \leq \frac{c}{\cosh \sigma} 
 $$
 \begin{proof}
  The proof consists of calculating the mean curvature of $ C_{\e/2} $ with respect to the ambient metric 
  \begin{align*} 
   \mathcal (\mathcal X \circ \Lambda_m)^* g_{eucl}(\zeta,\xi_3) & = 
   A^2(\Lambda_m(\zeta,\xi_3)) \left( d\zeta^2 + B^2(\Lambda_m(\zeta, \xi_3)) d \xi_3^2 \right) \\
   & = \frac{4}{ \left[|1-\zeta|^2 + (1 + |\zeta|^2)(\cosh( 2 \xi_3 )-1) 
   \right]^2} \left( d \zeta^2 + (1 + |\zeta|^2)^2 \, d \xi_3^2 \right) \\ \\
   & = a^2(\zeta,\xi_3) \left( d \zeta^2 + b^2(\zeta) \, d \xi_3^2  \right) = : g_m(\zeta,\xi_3)
  \end{align*}
 where $a (\zeta,\xi_3) = \dfrac{2}{|1-\zeta|^2 + (1 + |\zeta|^2)(\cosh(2 \xi_3)-1)}$ and $b(\zeta) = 1 + |\zeta|^2$. 
  
  \medskip
  
Let $\nabla^\e$ denote the Levi-Civita connection corresponding to this metric. We have the following estimates for the Christoffel symbols in a neighborhood of $(\zeta,\xi_3) = (0,0)$:
  $$ \begin{array}{l} \Gamma^1_{11} = - \Gamma^1_{22} = \Gamma^2_{12} = \frac{1}{a} \frac{\partial a}{\partial \xi_1} =  \mathcal O(1), \quad 
  \Gamma^2_{11} = - \Gamma^2_{22} = - \Gamma^1_{12} =  \frac{1}{a} \frac{\partial a}{\partial \xi_2} = \mathcal O(1) \\ \\
  \Gamma^1_{13} = \Gamma^2_{23} = \Gamma^3_{33} = \frac{1}{a} \frac{\partial a}{\partial \xi_3} = \mathcal O(\xi_3), \quad  
  \Gamma^3_{11} = \Gamma^3_{22} = - \frac{1}{ab^2} \frac{\partial a}{\partial \xi_3} = \mathcal O(\xi_3),
  \\ \\
  \Gamma^1_{33} = - (\frac{b^2}{a} \frac{\partial a}{\partial \xi_1} + b \frac{\partial b}{\partial \xi_1}) = \mathcal O(1), \quad
  \Gamma^2_{33} = - (\frac{b^2}{a} \frac{\partial a}{\partial \xi_2} + b \frac{\partial b}{\partial \xi_2}) = \mathcal O(1), \\ \\
  \Gamma^3_{13} = \frac{1}{a} \frac{\partial a}{\partial \xi_1} + \frac{1}{b} \frac{\partial b}{\partial \xi_1} = \mathcal O(1), \quad
  \Gamma^3_{23} = \frac{1}{a} \frac{\partial a}{\partial \xi_2} + \frac{1}{b} \frac{\partial b}{\partial \xi_2} = \mathcal O(1) \\ \\
  \Gamma^1_{23} = \Gamma^2_{13} = \Gamma^3_{12} = 0 \end{array} $$
  
  \medskip
  
Using $ |\zeta| = \e/2 \cosh \sigma, \ \xi_3 = \e/2 \, \sigma, $ and
$$ 
\nabla^\e_{\partial_p} \partial_q = \partial_p \, \partial_q \, X^{cat}_{\e/2} + \left[ \partial_p X^{cat}_{\e/2} \right]^i \left[ \partial_q  X^{cat}_{\e/2} \right]^j \Gamma^k_{ij} \, \partial_k, 
$$
where $ \partial_p $ and $ \partial_q $ stand for $ \partial_\sigma $ or $ \partial_\theta $ and $\partial_k = \partial_{\xi_k}, \, k=1,2,3$. We get
\begin{multline*}
\left| \left[ \nabla^\e_{\partial_p} \partial_q - \partial_p \, \partial_q \, X^{cat}_{\e/2} \right]^i (\sigma,\theta)  \right| \leq c \, \e^2 \cosh^2 \sigma, 
\quad i = 1,2 \\
\left| \left[ \nabla^\e_{\partial_p} \partial_q - \partial_p \, \partial_q \, X^{cat}_{\e/2} \right]^3 (\sigma,\theta) \right| \leq c \, \e^2 \cosh \sigma .
\end{multline*}
 
The unit normal to $C_{\e/2}$ with respect to the metric $g_m$ is 
$$ 
\mathcal N(\sigma,\theta) = \frac{1}{a \sqrt{\frac{b^2}{\cosh^2 \sigma} + \tanh^2 \sigma }} \left( - \frac{b}{\cosh \sigma} \, e^{i\theta}, \frac{1}{b} \tanh \sigma \right) .
$$
Using the the expansions for $a$ and $b$ in the neighbourhood of $ 0 $ and fact that the third coordinate of the vector $ \partial_p \, \partial_q \,  X^{cat}_{\e/2} $ is zero for all $p$ and $q$ we get the following expression for the second fundamental form :
$$ 
\mathfrak h_{\e}(\sigma,\theta) = \e(d\sigma^2 - d\phi^2) + h_{\e}(\sigma,\theta), 
$$
where $\left| \left( h_{\e} \right)_{pq}(\sigma,\theta) \right| \leq c \, \e^2 \cosh \sigma. $ On the other hand, we can write the expansion of the metric induced on $ C_{\e/2} $ from $g_m$ 
$$ 
\mathfrak g_{\e}(\sigma,\theta) = \e^2 \cosh^2 \sigma (d\sigma^2 + d\phi^2) + g_{\e}(\sigma,\theta), 
$$
where $ \left| \left( g_{\e} \right)_{pq}(\sigma,\theta) \right| \leq c \e^3 \cosh^3 \sigma$. 
  
\medskip  
  
Finally, $ \left| \mathcal H( \e/2 \, \cosh \sigma \, e^{ i \theta}) \right| = \left| \mathrm{tr} \left( \mathfrak g_\e^{-1} \mathfrak h_\e \right)(\sigma,\theta) \right| \leq \dfrac{c}{\cosh \sigma}. $
  
\end{proof}
\end{proposition}

\subsection{Catenoidal neck at $ z=0 $}

In a small neighborhood of $ z = \te $ take the change of variables $ z = r \, e^{i \phi}, \ \phi \in S^1 $. Then, in the neighbourhood of $ z = 0$, we have
$$ B(r \, e^{i \phi}) = B(r) = \frac{1}{2} + \mathcal O(r^2). $$ 
We remind that that $\tau \, G_n(z^n) + \tilde \tau \, \tilde G_n(z^n) = \tilde{\mathcal G}_n \, B $ is a function whose graph is close to the lower part of the euclidean catenoid scaled a factor $\te$. Then, the graph of $\tilde{ \mathcal G }_n$ is close to the lower part of the surface $ \tilde C_{\te} $, parametrized by
$$ 
\tilde X^{cat}_{\te} \, : \, (s,\phi) \in \mathbb R \times S^1 \, \mapsto \, (\te \cosh s \, e^{i\phi}, 2 \, \te s) .
$$

Let us define a cut-off function $r \, \mapsto \, \eta_{\te}^0(r)$, such that
$$ 
\eta_{\te}^0(r) \equiv 1 \quad \mbox{for} \quad r \in \left( 0, 1/2 \, \te^{1/2} \right), \quad \eta_{\te}^0(r) \equiv 0 \quad \mbox{for} \quad  r > 2 \, \te^{1/2} .
$$
Taking the change of variables: $ \te \cosh s \, e^{i \phi} = z = r \, e^{i\phi} $, for $ s>0 $ or $ s<0 $ we can parametrize the lower and the upper part of $ \tilde C_{\te}$ as vertical graphs 
$$ 
z \, \mapsto \, (z, \pm \, 2 \, G^{cat}_{\te}), 
$$ 
where 
$$ 
G_{\te}^{cat}(r) = \frac{\te}{2} \log \left( \frac{\te}{2r} \right) + \mathcal O \left( \frac{\te^3}{r^2} \right). 
$$

We define the function 
\begin{align*} 
\tilde \Upsilon :  \bar D^2 \, \setminus \{0, z_1, \ldots, z_m \} & \longrightarrow \mathbb R, \\
& \tilde \Upsilon (r,\phi) = (1 - \eta_{\te}^0(r)) \, \tilde{ \mathcal G_n } (r \, e^{i\phi}) + 2 \, \eta_{\te}^0 \, G_{\e}^{cat}(r), 
\end{align*} 
and parametrized $\tilde{\mathcal S_n} $ in the region 
$$ 
\Omega_{cat}^0 : = \left\{ \te \cosh s \, e^{i\phi} \, : \, \te \cosh s < 1/2 \, {\te}^{1/2}, \, \phi \in S^1  \right\} ,
$$
by $ (s,\phi) \, \mapsto \, \mathcal X \circ \tilde X^{cat}_{\te}(s,\phi) $ and as a bi-graph~: 
\begin{multline*}
(r,\phi) \, \mapsto \, \mathcal X \left\{ (r \, e^{i\phi}, \tilde \Upsilon(r,\phi)) \cup (r \, e^{i\phi}, - \tilde \Upsilon(r, \phi)) \right\}, \\[3mm]
\mbox{for} \quad z \in \Omega_{glu}^0 : = \left\{ r \, e^{i\phi} \, : \, 1/2 \, \te^{1/2} < r < 2 \, \te^{1/2}, \ \phi \in S^1 \right\}.
\end{multline*} 

We use the notations~: 
\begin{multline*} 
B(s) = B(\te \cosh s) = \frac{1}{2}(1 + \te^2 \cosh^2 s) \quad \text{and} \\ A(s) = A(\te \cosh s, 2 \,\te s) = \frac{1}{1 + B(s) \left( \cosh(2 \,\te s) - 1 \right)} .
\end{multline*}  

\begin{proposition} \label{catm}
There exists a constant $c$ which does not depend on $\e$ such that in the region 
$$ 
\Omega_{cat}^0 = \left\{ \te \cosh s \, e^{i\phi} \, : \, \te \cosh s < 1/2 \, {\te}^{1/2}, \, \phi \in S^1  \right\} ,
$$ 
the mean curvature of the surface $\tilde{\mathcal S_n}$ satisfies
$$ 
\left| \mathcal H(\te \, \cosh s \, e^{i \phi}) \right| \leq c \, \te^{1-\beta}, \quad \forall \beta \in (0,1). 
$$
\end{proposition} 
\begin{proof}
As in the proposition \ref{cat0} we would like to calculate the mean curvature of $ \tilde{C_\e} $ with respect to the ambient metric 
$$ 
\mathcal X^* g_{eucl}(z,x_3) = A^2(z,x_3) \left( dz^2 + B^2(z) dx_3^2 \right) = : \tilde g(z,x_3) .
$$
 
\medskip
  
We denote by $ \tilde \nabla^{\te}$ the Levi-Civita connection corresponding to this metric. Then, in the neighborhood of $(z,x_3) = (0,0)$ the Cristoffel symbols satisfy~:
  
  $$ \begin{array}{l} \tilde \Gamma^1_{11} = - \tilde \Gamma^1_{22} = \tilde \Gamma^2_{12} = \frac{1}{A} \frac{\partial A}{\partial {x_1}} =  \mathcal O(|z|
  x_3), \quad 
  \tilde \Gamma^2_{11} = - \tilde \Gamma^2_{22} = - \tilde \Gamma^1_{12} = \frac{1}{A} \frac{\partial A}{\partial {x_2}} = \mathcal O(|z| x_3) \\ \\
  \tilde \Gamma^1_{13} = \tilde \Gamma^2_{23} = \tilde \Gamma^3_{33} = \frac{1}{A} \frac{\partial A}{\partial {x_3}} = \mathcal O(x_3), \quad  
  \tilde \Gamma^3_{11} = \tilde \Gamma^3_{22} = - \frac{1}{AB^2} \frac{\partial A}{\partial x_3} = \mathcal O(x_3), \\ \\
  \tilde \Gamma^1_{33} = - (\frac{B^2}{A} \frac{\partial A}{\partial x_1} + B \frac{\partial B}{\partial x_1}) = \mathcal O(|z|), \quad
  \tilde \Gamma^2_{33} = - (\frac{B^2}{A} \frac{\partial A}{\partial x_2} + B \frac{\partial B}{\partial x_2}) = \mathcal O(|z|), \\ \\
  \tilde \Gamma^3_{13} = \frac{1}{A} \frac{\partial A}{\partial x_1} + \frac{1}{B} \frac{\partial B}{\partial x_1} = \mathcal O(|z|), \quad
  \tilde \Gamma^3_{23} = \frac{1}{A} \frac{\partial A}{\partial x_2} + \frac{1}{B} \frac{\partial B}{\partial x_2} = \mathcal O(|z|) \\ \\ 
  \tilde \Gamma^1_{23} = \tilde \Gamma^2_{13} = \tilde \Gamma^3_{12} = 0    \end{array} $$
  Using $ |z| = \te \cosh s, \ x_3 = 2 \, \te s, $ and
  $$ 
  \tilde \nabla^{\te}_{\partial_p} \partial_q = \partial_p \, \partial_q \, \tilde{X}_{\te}^{cat} + [\partial_p \tilde X_{\te}^{cat}]^i [\partial_q 
  \tilde{X}_{\te}^{cat}]^j \, \tilde \Gamma^k_{ij} \partial_k, 
  $$
  where $ \partial_p $ and $\partial_q$ stand for $\partial_s$ or $\partial_t$, we get
  \begin{multline*}
   \left| [ \tilde \nabla^{\te}_{\partial_p} \partial_q - \partial_p \, \partial_q \, \tilde X_{\te}^{cat}]^i (s,\theta)  \right| \leq c \, \te^3 \cosh^3 
   s, \quad i = 1,2 \\
\left| [\tilde \nabla^{\te}_{\partial_p} \partial_q - \partial_p \, \partial_q \, \tilde X_{\te}^{cat}]^3 (s,\theta) \right| \leq c \, \te^{3-\beta} \cosh^2s, \quad \forall \beta \in (0,1) .
  \end{multline*}
The normal vector field to $\tilde C_{\te}$ with respect to the metric $\mathcal X^*g_{eucl}$ is 
\begin{equation*} 
\tilde{\mathcal N}(s,\phi) = \frac{1}{ A \sqrt{ \frac{4B^2}{\cosh^2 s} + \tanh^2 s } } \, \left( - \frac{2 B}{\cosh s} \, e^{i\phi}, \frac{1}{B} \tanh s  \right).   \end{equation*}

As the third coordinate of the vector $ \partial_p \, \partial_q \, \tilde{X}_{\te}^{cat} $ is zero for all $p$ and $q$, we get the following expression for the second fundamental form~:
$$ 
\tilde{\mathfrak h}_{\te}(s,\phi) = \te \, (ds^2 - d\phi^2) + \tilde{h}_{\te}(s,\phi), 
$$
where $  \left| \left( \tilde{h}_{\te} \right)_{pq}(s,\phi) \right| \leq c \, \te^{3-\beta} \cosh^2 s. $ On the other hand the metric induced on $C_{\te}$ from $ \mathcal X^*g_{eucl}$ can be written as
$$ 
\tilde{\mathfrak g}_{\te}(s,\phi) = \te^2 \cosh^2 s(ds^2 + d \phi^2) + \tilde{g}_{\te}(s,\phi), 
$$
where $ \left| \left( \tilde{g}_{\te} \right)_{pq}(s,\phi) \right| \leq c \, \te^{4-\beta} \, \cosh^2 s$. Finally, 
$$ 
\left| \mathcal H(\te \, \cosh s \, e^{ i \phi} ) \right| = \left| \mathrm{tr} \left( \tilde{ \mathfrak g}_{\te}^{-1} \tilde{\mathfrak h}_{\te} \right)(s,\phi) \right| \leq c \, \te^{1 - \beta}, \quad \forall \beta \in (0,1). 
$$
 \end{proof}
 
 \subsection{The graph region}
 
 Away from the catenoidal bridges and the catenoidal neck, that is in the region
 $$ 
 \Omega_{gr}: = \left\{ z \in D^2 \, : \, 2 \, \te^{1/2} < |z| < 1 \right\} \setminus \underset{m=1}{\overset{n}{\cup}} \lambda_m \left\{ \zeta \in \mathbb 
 C_- \, : \, |\zeta| < 2 \, \e^{2/3} \right\} ,
 $$
 we parametrize the surface $\tS$ as a bi-graph
 $$ 
 \mathcal X \left\{ ( z, \tilde{\mathcal G}_n(z) ) \cup ( z, - \tilde{\mathcal G}_n(z) ) \right\}. 
 $$
 \begin{proposition} \label{graph}
 There exists a constant $c$ which does not depend on $\e$, such that in the region 
 \begin{multline*} 
 \Omega_{gr} \cup \Omega_{glu}^0 \underset{m=1}{\overset{n}{\cup}} \Omega_{glu}^m \\[3mm] 
 = \left\{ z \in D^2 \, : 1/2 \, \te < |z| < 1 \right\} \setminus \underset{m=1}{\overset{n}{\cup}} \lambda_m \left\{ \zeta \in \mathbb C_- \, : \, |\zeta| < 
 1/2 \, \e^{2/3} \right\},
 \end{multline*}
 the mean curvature $\mathcal H$ of $ \tS $ satisfies
 \begin{equation} \label{MCgr} 
 \left| \mathcal H(z) \right| \leq c \, \e^{3 - \beta} \left( \dfrac{1}{|z|^4} + \sum_{m=1}^n \dfrac{1}{|z-z_m|^4} \right), \quad \forall \beta \in (0,1) .
 \end{equation}
 \end{proposition}

\begin{proof}
According to the lemma \eqref{meancurvatureGraph}, the mean curvature of the graph $ \mathcal X \left( z, u(z) ) \right), \,  u \in \mathcal C^2(D^2)$ satisfies~:
 \begin{align*} 
  H(u) & =  \frac{1}{A^3(u) \, B} \, \mbox{\rm div} \, \left( \frac{A^2(u) \, B^2 \, \nabla u}{\sqrt{1+ B^2\, |\nabla u|^2}} \right) + 2 \, \sqrt{1+ B^2\, |
  \nabla u|^2} \,  \sinh u \\ \\
  & = \frac{2}{A^2(u)} \frac{B \nabla u \nabla A(u)}{\sqrt{ 1 + B^2 |\nabla u|^2 }} + \frac{2}{A(u)} \frac{\nabla B \nabla u}{\sqrt{ 1 + B^2 |\nabla u|^2 }} + 
  \frac{1}{A(u)} \frac{B \Delta u}{\sqrt{ 1 + B^2 |\nabla u|^2 }} \\ \\
  & - \frac{1}{A(u)}\frac{B^2 \nabla B \nabla u |\nabla u|^2}{( 1 + B^2 |\nabla u|^2 )^{3/2}} - \frac{1}{2A(u)} \frac{B^3 \mathrm{Hess_u}(\nabla u, \nabla u)}{( 1 + B^2 |\nabla 
  u|^2 )^{3/2}} + 2 \sqrt{1 + B^2 |\nabla u|^2} \sinh u \\ \\
  & = \Delta \, (B u) + P_3(u, \nabla u, \nabla^2 u)
 \end{align*}
where $P_3$ is a bounded nonlinear function which can be decomposed in entire series in $u$ and the components of $\nabla u$ and $\nabla^2 u$ for $ \|u\|_{\mathcal C^1} \leq 1 $ with terms of lowest order 3 and where the components of $\nabla u$ appear with an even power and the components of $\nabla^2 u$ only with the power $1$. 
\medskip
\newline In the region $\Omega_{gr}$ the surface $ \tS $ is parametrized as a graph of one of the functions $ \pm \, \tilde{\mathcal G}_n $, where
$$ 
B(z) \, \tilde{\mathcal G}_n(z) = - \frac{\e n}{2} + \e \mathrm{Re}(f_n(z)) + \te \log |z| + \te , 
$$
Studying the behaviour of the function $f_n(z)$ one can easily verify that
\begin{multline*} 
 \left| \tilde{\mathcal G}_n(z) \right| \leq c \e \, \left( | \log \e | + \left| \log|z| \right| + \sum_{m=1}^n \left| \log|z-z_m| \right| \right), \\[3mm] 
 \left| \nabla \tilde{\mathcal G}_n(z) \right| \leq c \e \, \left( \frac{1}{|z|} + \sum_{m=1}^n \frac{1}{|z-z_m|} \right), \\[3mm] 
 \left| \nabla^2 \tilde{\mathcal G}_n(z) \right| \leq c \e \, \left( \frac{1}{|z|^2} + \sum_{m=1}^n \frac{1}{|z-z_m|^2} \right) 
\end{multline*}
As $ \Delta \left( B \tilde{\mathcal G}_n \right) = 0 $, analysing carefully the terms in $P_3( \tilde{\mathcal G}_n, \nabla \tilde{\mathcal G}_n, \nabla^2 \tilde{\mathcal G}_n)$ one can see that \eqref{MCgr} is true in $\Omega_{gr}$.

\medskip

In the regions $\Omega_{glu}^m$ the surface $\tS$ is parametrized as a graph of one of the functions $ \pm 2 \, \Upsilon^m $, where
$$ 
\Upsilon^m(\lambda_m(\rho/2 \, e^{i\theta})) = \, \bar \Upsilon(\rho,\theta) =  (1 - \bar \eta_\e(\rho)) \, \frac{1}{2} \, \bar { \mathcal G }_n(\rho,\theta) + \, \bar \eta_\e(\rho) \, G_{\e/2}^{cat}(\rho). 
$$
In the neighbourhood of $\zeta = 0$ we have: 
$$ 
\nabla_z = \frac{1}{2} \, \nabla_\zeta ( 1 + \mathcal O(|\zeta|) ), \quad \nabla^2_z = \frac{1}{4} \, \nabla^2_\zeta(1 + \mathcal O(|\zeta|)) .
$$
Using that $  \, \left| \lambda_m(\rho/2 \, e^{i \theta}) - z_m \right| = \rho + \mathcal O(\rho^2) \, $, we obtain

\begin{align*}  
  & \bar { \mathcal G }_n \sim G_{\e/2}^{cat} = \mathcal O(\e \log \e), \quad \bar { \mathcal G}_n - G_{\e/2}^{cat} = \mathcal O(\e^{1 - \beta} \rho ), 
  \\[3mm] 
  & \left| \nabla \bar { \mathcal G}_n \right| \sim \left| \nabla G_{\e/2}^{cat} \right| = \mathcal O \left( \frac{\e}{\rho} \right)  \quad \left| \nabla 
  (\bar { \mathcal G}_n - G_{\e/2}^{cat}) \right| = \mathcal O(\e^{1-\beta}), \\[3mm]
  & \left| \nabla^2 \bar { \mathcal G }_n \right| \sim \left| \nabla^2 G_{\e/2}^{cat} \right| = \mathcal O \left( \frac{\e}{\rho^2} \right), \quad \left| 
  \nabla^2( \bar { \mathcal G}_n + G_{\e/2}^{cat} ) \right| = \mathcal O(\e^{1/3 - \beta})  
 \end{align*}
We introduce the function
$$ 
\bar B(\rho,\theta) = B \left( \lambda_m \left( \dfrac{\rho}{2} \, e^{i \theta} \right) \right) = \frac{ 1 + \left| \left( 1 + \frac{\rho}{2} \, e^{i\theta} \right) / \left( 1 - \frac{\rho}{2} \, e^{i \theta} \right) \right|}{2} = 1 + \mathcal O(\rho). 
$$
We have
$$ 
\partial_\rho \bar B \sim \partial_\rho^2 B \sim \partial_\rho \partial_\theta \bar B = \mathcal O(1), \quad \partial_\theta \bar B \sim \partial_\theta^2 \bar B = \mathcal O(\rho) .
$$
\newline On the other hand, the cut-off function $\bar \eta_\e$ satisfies
$$ 
\bar \eta_\e = \mathcal O(1), \quad \dfrac{d \bar \eta_\e}{d \rho} = \mathcal O \left( \frac{1}{\rho} \right), \quad \dfrac{ d^2 \bar \eta_\e}{d \rho^2 } = \mathcal O \left( \frac{1}{\rho^2} \right). 
$$
Using that
$$ 
\nabla \bar \Upsilon = \frac{1}{2}(1 - \bar \eta_\e) \, \nabla \bar { \mathcal G }_n + \bar \eta_\e \nabla G_{\e/2}^{cat} + \nabla \bar \eta_\e \left( G_{\e/2}^{cat} - \frac{1}{2} \bar {\mathcal G }_n \right) , 
$$
\begin{align*} 
 \nabla^2 \bar \Upsilon & = \frac{1}{2} ( 1 - \bar \eta_\e) \, \nabla^2 \bar { \mathcal G }_n + \bar \eta_\e \nabla^2 G_{\e/2}^{cat} \\ 
 & + 2 \nabla \bar \eta_\e \left( \nabla \left( G_{\e/2}^{cat} - \frac{1}{2} \bar {\mathcal G}_n \right) \right)^{t} + \nabla^2 \bar \eta_e \left( G_{\e/2}^{cat} - \frac{1}{2} \bar{ \mathcal G }_n \right) ,
\end{align*}
we get the estimates
$$ 
\left| \nabla \bar \Upsilon \right| = \mathcal O \left( \frac{\e^{1- \beta}}{\rho} \right), \quad \left| \nabla^2 \bar \Upsilon \right|  = \\ \mathcal O \left( \frac{\e^{1 - \beta}}{\rho^2} \right) ,
$$
and 
$$ 
P_3(\Upsilon) = \mathcal O \left( \frac{\e^{3 - \beta}}{\rho^4} \right) = \mathcal O(\e^{1/3 - \beta}) ,
$$
for all $\beta \in (0,1)$. We also have 
$$ 
\Delta_z = \frac{1}{4} \, |1-\zeta^2| \, \Delta_\zeta, 
$$ 
where $ \Delta_z $ and $ \Delta_\zeta $ are Laplacian operators in coordinates $z$ and $\zeta$. So,
$$ 
\Delta \left( \bar B \, G_{\e/2}^{cat} \right) = G_{\e/2} \, \Delta \bar B + 2 \nabla \bar B \, \nabla G_{\e/2}^{cat} + \bar B \, \Delta 
G_{\e/2}^{cat} = \mathcal O \left( \frac{\e^3}{\rho^4} \right).
$$
Putting this calculations together, we check that in $ \Omega_{glu}^m $ the mean curvature of $\tS$ satisfies
$$ 
\mathcal H = \mathcal O \left( \e^{1/3 - \beta} \right), \quad \forall \beta \in (0,1). 
$$ 
An identical proof shows that in $\Omega_{glu}^0$ we have $ \mathcal H = \mathcal O(\e^{1-\beta})$.
 \end{proof}

\section{Perturbations of $\tS$}

Recall that the surface $\tilde{\mathcal S_n}$ can be seen as an image by the mapping $\mathcal X$ of a surface $\tilde S_n$ (constructed in the previous paragraph) which is contained in the unit cylinder $D^2 \times \mathbb R$.  We would like to calculate the mean curvature of small perturbations of $\tilde{\mathcal S_n}$ and to this end we calculate the mean curvature of small perturbations of the surface $\tilde S_n$ with respect to the metric $\tilde g = \mathcal X_* g_{eucl}$. 

\medskip

Let, as before, $ \tilde{\mathcal N} $ denote the unit normal vector field to $\tilde C_{\te}$ with respect to the metric 
$$
\tilde g(z,x_3) = A^2(z,x_3) \left( dz^2 + B^2(z) \, dx_3^2 \right). 
$$
and take a function $ w \in \mathcal C^2(\tilde S_n) $ small enough, invariant under rotations by the angle $\frac{2\pi}{n}$ and the transformation $ z \mapsto \bar z $. 
We denote $\tilde S_n(w)$ the surface parametrized by
$$ 
(s,\phi) \in \mathbb R \times S^1 \, \mapsto \, \tilde X_{\te}^{cat}(s,\phi) + w(s,\phi) \, \tilde{\mathcal N}(s,\phi) ,
$$
in region $\Omega_0^{cat}$. Furthermore, in the region
\begin{align*} 
 \Omega_{gr} \cup \Omega_{glu}^0 = \left\{ z \in D^2 \, : \, 1/2 \, \te^{1/2} < |z| < 1 \right\} \setminus \overset{n}{\underset{m=1}{\cup}} \lambda_m \left\{ \zeta \in \mathbb C_- \, : \, |\zeta| < 2 \, \e^{2/3} \right\} ,
\end{align*}
we parametrize $S_n(w)$ by
$$ 
z \, \mapsto \, \left( z,  \pm \, \tilde \Upsilon(z) \right) \pm w(z) \, \tilde \Xi_{\te}(z) , \quad \mbox{where} \quad \tilde \Xi_{\te} = (1 - \eta_{\te}^0) \, \, \partial_{x_3} +  \eta_{\te}^0 \, \frac{1}{2} \, \tilde{\mathcal N}, 
$$
where $ \eta_{\te}^0(|z|) $ is the cut-off function defined in the paragraph 6.2. Notice that in $\Omega_{glu}^0$
$$ 
\| \tilde \Xi_{\te} \|_{\tilde g} \sim \| \partial_{x_3} \|_{\tilde g} = \dfrac{1}{2} + \mathcal O(\te).
$$

\medskip

In the neighbourhood of the $n$-th root of unity $z_m$ the surface $\tilde S_n$ can be seen as an image by the mapping $\Lambda_m$ of a surface $\bar S_n$ contained in $\mathbb C_- \times \mathbb R$. We put  $ \bar w = w \circ \lambda_m \in \mathcal C^2(\bar S_n) $ and parametrize $\tilde S_n(w)$ in the region $\Omega_{cat}^m$ by 
$$ 
(\sigma, \theta) \in \mathbb R \times \left[ \pi/2, 3\pi/2 \right] \, \mapsto \, \Lambda_m \left( X_{\e/2}^{cat}(\sigma,\theta) + \dfrac{a}{2} \, \bar w \, \mathcal N(\sigma, \theta) \right) ,
$$
where $ \mathcal N $ is the unit normal vector field to the half-catenoid $C_{\e/2}$ with respect to the metric 
$$ 
g_m(\zeta,\xi_3) = (\Lambda_m \circ \mathcal X)_* g_{eucl} = a^2(\zeta,\xi_3) \left( d \zeta^2 + b^2(\zeta) \, d \xi_3^2 \right) .
$$
In the regions $\Omega_{glu}^m $, we parametrize $\tilde S_n(w)$ by
$$ 
\zeta \, \mapsto \, \Lambda_m \left( \left( \zeta, \pm \, \bar \Upsilon(\zeta) \right) \, \pm \, \bar w(z) \, \bar \Xi_{\e}(\zeta) \right), \quad \mbox{where} \quad \bar \Xi_{\te} = \frac{1}{2} \left( ( 1 - \bar \eta_\e) \, \partial_{\xi_3} + \bar \eta_\e \, a \, \mathcal N \right) ,
$$
and $ \bar \eta_\e(|\zeta|)$ be the cut-off function, introduced in the paragraph 6.1. Notice that in $ \Omega_{glu}^m $, we have
$$ 
\| \Xi_{\te} \|_{g_m} \sim \dfrac{1}{2} \| \partial_{\xi_3} \|_{g_m} = 1 + \mathcal O(\e^{4/3}) .
$$

\begin{remark}
We multiply the vector field $\mathcal N$ by the factor $a$ in order to make the vector field
$$
\partial_\theta \left( a \, \mathcal N \right)(\zeta) \left. \right|_{\theta \in \left\{ \frac{\pi}{2}, \frac{3\pi}{2} \right\}} ,
$$
horizontal. In this case, the condition sufficient for $\tS(w)$ to be orthogonal to the unit sphere is
$$ 
\left. \partial_\theta \left( \bar w \right) \right|_{\theta \in \left\{ \frac{\pi}{2}, \frac{3\pi}{2} \right\}} = 0 .
$$
\end{remark}
 
\begin{notation}
Let $\Omega$ denote a coordinate domain we work in. From now on, when we don't need a more detailed information, we use the following notations~:
\begin{itemize}
 
\item $L$ for any bounded second order linear differential operator defined in $\Omega$ (in other words $L \, w$ is a linear combination of $w$ and the components of $\nabla w$ and $\nabla^2 w$ with coefficients which are bounded functions in $\Omega$, where $\nabla$ and $\nabla^2$ are the gradient and the Hessian in the chosen coordinates).
 
\smallskip
 
 \item $Q^k(w, \nabla w, \nabla^2 w)$, $k \in \mathbb N$, for any nonlinear function, which can be decomposed in entire series with terms of lowest order $k$, and where the components of $\nabla^2 w$ appear only with power 1. We will also use the notation $Q^k(w)$ for brevity.
 
\smallskip
 
\item Let ${\gamma} \in \mathcal C^{\infty} (\Omega)$ be a positive real function. We denote $ L^\gamma \, w $ and $ Q^{k,\gamma}(w) $ functions which share the same properties as $ L \, w $ and $ Q^k(w) $ with the only difference that the components of the gradient and the Hessian of $w$ are calculated with respect to the metric $ \gamma^{-2} \, g_{eucl} $.

\smallskip

For example, if we work in the coordinates $(r,\phi)$ and take $ \gamma = r $, then $L^\gamma$ will be a linear differential operator in $ r^2 \partial^2_r$, $\partial^2_\phi$, $r \partial^2_{r \phi}$, $r \partial_r$ and $\partial_\phi$.

 \end{itemize}

\end{notation}

\subsection{Mean curvature of the perturbed graph} 

In the region 
$$ 
\Omega_{gr} = \{ z \in D^2 \, : \, 2 \, \te^{1/2} < |z| < 1 \} \setminus \underset{m=1}{\overset{n}{\cup}} \lambda_m \, \{ \zeta \in \mathbb C_- \, : \, |\zeta| < 2 \, \e^{2/3} \} ,
$$
we suppose that $ \| w \|_{\mathcal C^2(\Omega_{gr})} < 1 $. Then, the mean curvature of $ \tS(w) $ satisfies 
$$ 
\mathcal H(w) = \mathcal H(0) + \Delta ( B \, w ) + P_3 \left( \tilde{\mathcal G}_n + w \right). 
$$
Analysing carefully the the terms of $P_3$ and the expansion of the function $ \tilde{\mathcal G}_n $, we get 
\begin{equation} \label{MCgrPert} \begin{array}{ll} 
 \mathcal H(w) & = \mathcal H(0) + \Delta ( B \, w ) \\[3mm] 
 & + \mathrm{max} \left\{ \left| \nabla \tilde{\mathcal G}_n \right|^2, \, \left| \nabla^2 \tilde{\mathcal G}_n \right| \left| \nabla \tilde{\mathcal G}_n 
 \right|, \, \tilde{\mathcal G}_n^2, \, \left| \tilde{\mathcal G}_n \right| \left| \nabla^2 \tilde{\mathcal G}_n \right|, \left| \tilde{\mathcal G}_n \right| \left| 
 \nabla \tilde{ \mathcal G}_n \right| \right\} \, L \, w \\[3mm]
 & + \mathrm{max} \left\{ \left| \tilde{\mathcal G}_n \right|, \, \left| \nabla \tilde{\mathcal G}_n \right|, \, \left| \nabla^2 \tilde{\mathcal G}_n \right| 
 \right\} \, Q_2(w) + Q_3(w) \end{array}
\end{equation} 
We introduce the weight functions
$$ 
\gamma_0(z) = |z|, \quad \gamma_m(z) = |z-z_m|, \ m = 1, \ldots, n \quad \mbox{and} \quad \gamma(z) = |z| \, \overset{n}{ \underset{m=1}{\Pi} } |z-z_m| .
$$ 
Then, (using the relation between $\e$ and $\te$) \eqref{MCgrPert} can be written as 
$$ 
\mathcal H(w) = \mathcal H(0) + \Delta ( B \, w) + \dfrac{\e^{2-\beta}}{\gamma^4} \, L^{\gamma}_z \, w + \dfrac{\e^{1-\beta}}{\gamma^4} \, Q_z^{2,\gamma}(w) + \dfrac{\e^{-\beta}}{\gamma^4} \, Q_z^{3,\gamma}(w) .
$$
for all $\beta \in (0,1)$. (We use the lower index to indicate the coordinate system we work in). On the other hand, in the neighbourhood of $ z = 0 $ we have
$$ 
\mathcal H(w) = \mathcal H(0) + \Delta ( B \, w ) + \frac{\te^2}{|z|^4} \, L^{\gamma_0}_z \, w + \frac{\te}{|z|^4} \, Q_z^{2,\gamma_0}(w) + \frac{1}{|z|^4} \, Q_z^{3,\gamma_0}(w) ,
$$
and in the neighbourhood of $z = z_m$ 
\begin{multline*} 
 \mathcal H(w) = \mathcal H(0) + \Delta ( B \, w ) + \frac{\e^{2 - \beta}}{|z-z_m|^4} \, L^{\gamma_m}_z \, w + 
 \frac{\e^{1 - \beta}}{|z-z_m|^4} \, Q_z^{2,\gamma_m}(w) + \frac{\e^{-\beta}}{|z-z_m|^4} \, Q_z^{3,\gamma_m}(w).
\end{multline*} 
for all $ \beta \in (0,1) $ (where we used $ \dfrac{1}{\gamma} < \dfrac{n}{\gamma_m} $ in the neighborhood of $z_m$).

\subsection{Mean curvature of the perturbed neck}

In the region 
$$ 
\Omega_{cat}^0 = \{ z \in D^2 \, : \, \te < |z| < 1/2 \, \te^{1/2} \} ,
$$
the surface $\tilde S_n(w)$ is parametrized as a normal graph around $\tilde C_{\te}$ for the function $u = 1/2 \, w $. We suppose that
\begin{equation} \label{CondNormw} 
 \left\| \dfrac{w}{\te \cosh s} \right\|_{ \mathcal C^2( (-s_{*}, s_{*}) \times S^1 )} \leq 1, \quad \te \, \cosh s_* = 1/2 \, \te^{1/2}
\end{equation}

The tangent space to $\tilde S_n(w)$ is spanned by the vector fields
$$ 
\tilde T_s(u) = \tilde T_s + \partial_s u \, \tilde{\mathcal N} + u \, \partial_s \tilde{\mathcal N}, \quad \tilde T_\phi(u) = \tilde T_\phi + \partial_\phi u \, \tilde{\mathcal N} + u \, \partial_\phi \tilde{\mathcal N}, 
$$
and let us choose functions $ \nu, \varkappa, \mu \in \mathcal C^{\infty}(\mathbb R)$, such that $ \nu(0) = \varkappa(0) = \mu(0) = 0 $ and
$$ 
\tN(u) = \tN + \nu(u) \, \tN +  \varkappa(u) \, T_s + \mu(u) \, \tilde T_\phi, 
$$
is the normal unit vector field to $\tilde S(w)$. We have
\begin{equation} \label{CoefNorm} 
\tilde g(u) \left( \tN(u), \tilde T_s(u) \right) = 0, \quad \tilde g(u) \left( \tN (u), \tilde T_\phi(u) \right) = 0, \quad \tilde g(u)  \left( \tN(u),\tN(u) \right) = 1.
\end{equation}
where $\tilde g(u) $ is the scalar product corresponding to the metric $\tilde g$ taken along $\tilde S(w)$. Using the expression for $\tilde g$, we get
$$ 
\tilde g (u)(s,\phi) - \tilde g(s,\phi) = \left ( \te^{1 - \beta} \, L_{s,\phi} \, u + Q^2_{s,\phi}(u) \right) \tilde g(s,\phi) + \left(  \te \, L_{s,\phi} \, u  + \frac{1}{\cosh^2 s} \, Q^2_{s,\phi}(u) \right) dx_3^2 , 
$$ 
and from \eqref{CoefNorm}, we deduce that
$$ \begin{array}{l} 
 \nu(u) = \te^{1 - \beta} \, L_{s,\phi} \, u + Q^2_{s,\phi}(u) \\[3mm]
 \varkappa(u) = - \dfrac{1}{\te^2 \cosh^2 s} \partial_s u + \dfrac{1}{\cosh^2 s} \, L_{s,\phi} \, u + \dfrac{\te^{1-\beta}}{\te^2 \cosh^2 s} \, Q^2_{s,\phi}(u) 
 \\[3mm]
 \mu(u) = - \dfrac{1}{\te^2 \cosh^2 s} \partial_\phi u + \dfrac{1}{\cosh^2 s} \, L_{s,\phi} \, u + \dfrac{\te^{1-\beta}}{\te^2 \cosh^2 s} \, Q^2_{s,\phi}(u) ,
 \end{array} 
$$
\newline where we used the fact that $ dx_3^2(\tilde T_s,\tN) $ and $ dx_3^2(\tilde T_\phi, \tN) $ can be bounded by a constant times $\te$ and the estimate 
$$ 
\left| \tilde g(\tilde T_p, \tilde T_q) - \te^2 \cosh^2 s \ \delta_{pq} \right| \leq c \, \e^3, 
$$ 
where $\tilde T_p$ and $\tilde T_q$ stand for $\tilde T_s$ or $\tilde T_\phi$. We can write the normal vector field to $\tilde S(w)$ in the form
\begin{multline*} 
 \tN(u) = \tN - \frac{1}{\te^2 \cosh^2 s} \left( \partial_s u \, \tilde T_s + \partial_\phi u \, \tilde T_\phi \right) \\
 + \left[ \te^{1 - \beta} \, L_{s,\phi} \, u + Q^2_{s,\phi}(u) \right]^\bot + \left[ \frac{\te}{\cosh s} \, L_{s,\phi} \, u + \frac{\te^{- \beta}}{ \cosh s} \, 
 Q^2_{s,\phi}(u) \right]^T 
\end{multline*}
where $[*]^\bot$ and $[**]^T$ denote a normal and a tangent vector fields of norm $*$ and $**$.

\medskip

We denote $ \tilde \nabla^{\te}(u) $ and $\tilde \Gamma_{ij}^k(u)$ the Levi-Civita connection and the Cristoffel symbols corresponding to the metric $\tilde g$ and taken along the surface $\tilde S(w)$. Then, we have
$$ 
\tilde \Gamma_{ij}^k(u) = \tilde \Gamma_{ij}^k + L_{s,\phi} \, u + Q^2_{s,\phi}(u) ,
$$
 \begin{multline*} 
  \tilde \nabla^{\te}_{\partial_p}\partial_q (u) = \tilde \nabla^{\te}_{\partial_p}\partial_q + \partial_p \partial_q u \, \tN + 
  \partial_p u \, \partial_q \tN  + \partial_q u \, \partial_p \tN + u \, \partial_p \, \partial_q \tN \\[3mm]
  + \te^2 \cosh^2 s \, L_{s,\phi} \, u +  \te \cosh s \, Q^2_{s,\phi}(u)
 \end{multline*}
where $\partial_p$ and $\partial_q$ stand for $\partial_\phi$ or $\partial_s$. This allows us to find the second fundamental form of the surface $\mathcal S_n(w)$ :
$$ 
\left( \tilde{\mathfrak h}_{\te}(u) \right)_{pq} = \tilde g(u) \left( \nabla^{\te}_{\partial_p} \partial_q(u), \tN(u) \right).  
$$
Note that
$$ 
\ \left| \tilde g(\partial_p \, \tN, \partial_q \, \tN) - \frac{1}{\cosh^2 s} \, \delta_{pq} \right| \leq c \, \te. 
$$ 
Putting all the estimates together, we obtain
\begin{align*}
  \tilde{\mathfrak h}_{\te}(u)(s,\phi) & = \tilde{\mathfrak h}_\e(s,\phi) + \left( \begin{array}{cc} \partial_s^2 u & \partial_s \partial_\theta u \\ 
  \partial_s \partial_\phi u & \partial_\phi^2 u \end{array} \right) \\[3mm] 
  &- \frac{u}{\cosh^2 s} \left( \begin{array}{cc} 1 & 0 \\ 0 & 1 
  \end{array} \right) + \tanh s \left( \begin{array}{cc} - \partial_s u & \partial_\phi u \\ \partial_\phi u & \partial_s u 
  \end{array} \right) \\[3mm]
  & + \left( \te^{2 - \beta} \, L_{s,\phi} \, u + \te^2 \cosh^2 s \, L_{s,\phi} \, u + \frac{1}{\te \cosh^2 s} \, Q^2_{s,\phi}(u) \right) \tilde{\mathfrak 
  h}_0(s,\phi)
 \end{align*}
where $\tilde{\mathfrak h}_0$ is a bounded symmetric 2-form, which does not depend on $\te$. On the other hand, the first fundamental form $\tilde{\mathfrak g}_{\te}(u)$, which corresponds to the metric induced on $ \mathcal S_n(w) $ from $\tilde g$, satisfies
$$ 
\tilde{\mathfrak g}_{\te}(u) = \tilde{\mathfrak g}_{\te} - 2 u \, \tilde{\mathfrak h}_{\te} + Q^2_{s,\phi}(u) .
$$
This yields
\begin{align*}
 \frac{\mathrm{det}(\tilde{\mathfrak g}_{\te}(u))}{\mathrm{det}(\tilde{\mathfrak g}_{\te})} \tilde{\mathfrak g}_{\te}^{-1}(u) & = \tilde{\mathfrak 
 g}_{\te}^{-1} + \frac{2 \e u}{\te^4 \cosh^4 s} \left( \begin{array}{cc} 1 & 0 \\ 0 & -1 \end{array} \right) \\[3mm]
 & + \left( \frac{1}{\te \cosh^4 s} \, L_{s,\phi} u + \frac{1}{\te^4 \cosh^4 s} Q^2_{s,\phi} (u) \right) \tilde{\mathfrak g}_0
\end{align*}
\newline where $\tilde{\mathfrak g}_0$ is a bounded $2$-form. Going back to $ w = 2 \, u  $, we obtain
\begin{equation} \label{MCcatBrid}
 \begin{array}{ll} \mathcal H(w) = \,
  & \mathcal H(0) + \dfrac{1}{2} \, \dfrac{1}{\te^2 \cosh^2 s} \left( \partial_s^2 + \partial_\phi^2 + \dfrac{2}{\cosh s} \right) \, w \\[5mm] 
  & + \left( 1 + \dfrac{\te^{- \beta}}{\cosh^2 s} \right) \, L_{s,\phi} \, w + \dfrac{1}{\te^3 \cosh^4 s} \, Q^2_{s,\phi}(w) + \dfrac{1}{\te^4 \cosh^4 s} \, 
  Q^3_{s,\phi}(w) .\\[5mm]
  \end{array}
\end{equation} 

\medskip

\subsection{Mean curvature of the perturbed bridges}

In the region 
$$ 
\Omega_{cat}^m = \lambda_m \{ \zeta \in \mathbb C_- \, : \, \e < |\zeta| < 1/2 \, \e^{2/3} \} \, ,
$$ 
the surface  $ \tilde S_n(w) $ is parametrized as the image by the mapping $\Lambda_m $ of the normal graph about $ C_{\e} $ for the function $ \bar u = a \, \bar w $, scaled by the factor $\frac{1}{2}$. We suppose that
\begin{equation} \label{CondNormw} 
 \left\| \dfrac{\bar w}{\e \cosh \sigma} \right\|_{ \mathcal C^2( (-\sigma_{*}, \sigma_{*}) \times [ \pi/2, 3 \pi/2 ] )} \leq 1, \quad \e \cosh \sigma_* = 1/2 
 \, e^{2/3}
\end{equation}
Our goal is to calculate the mean curvature of $ S(\bar w)$ with respect to the metric 
$$ 
g_m = a^2(d\zeta^2 + b^2  d\xi_3^2 ).
$$ 
The computation is very similar to the one we have done in the previous paragraph and we only need to change several estimates. The scalar product along $ S(\bar w)$ satisfies
$$ 
g_m(\bar u)(\sigma,\theta) - g_m(\sigma,\theta) = \left( L_{\sigma,\theta} \, \bar u + Q^2_{\sigma,\theta}(\bar u) \right) \, g_m(\sigma,\theta) + \left( \e L_{\sigma,\theta} \, \bar u + \frac{1}{\cosh^2 \sigma} \, Q^2_{\sigma,\theta}(\bar u) \right) \, d\tau^2 .
$$
Then, the normal vector field to $\mathcal S_n(\bar w)$ can be written as
\begin{multline*} 
 \mathcal N(\bar u) = \mathcal N - \frac{1}{\e^2 \cosh^2 \sigma} \left( \partial_\sigma \bar u \, T_\sigma + \partial_\theta \bar u \, T_\theta \right) \\
 + \left[ L_{\sigma,\theta} \, \bar u + Q^2_{\sigma,\theta}(\bar u) \right]^\bot + \left[ \frac{\e}{\cosh \sigma} \, L_{\sigma,\theta} \, \bar u + \frac{1}{\e 
 \cosh \sigma} \, Q^2_{\sigma,\theta}(\bar u) \right]^T ,
\end{multline*}
and the components of the Levi-Civita connection are
\begin{multline*} 
 \nabla^\e_{\partial_\alpha}\partial_\beta (\bar u) = \nabla^\e_{\partial_\alpha}\partial_\beta + \partial_\alpha \partial_\beta \bar u \,\mathcal N + 
 \partial_\alpha \bar u \, \partial_\beta \mathcal N + \partial_\beta \bar u \, \partial_\alpha \mathcal N + \bar u \, \partial_\alpha \, \partial_\beta 
 \mathcal N  \\[3mm]
  + \e \cosh \sigma \, L_{\sigma,\theta} \, \bar u + Q^2_{\sigma,\theta}(\bar u).
 \end{multline*}
The first and the second fundamental forms satisfy~:
\begin{align*}
 \frac{\mathrm{det}(\mathfrak g_\e(\bar u))}{\mathrm{det}(\mathfrak g_\e)} \mathfrak g_\e^{-1}(\bar u) & = \mathfrak g_\e^{-1} + \frac{\e \bar u}{\e^4  \cosh  
 ^4 \sigma} \left( \begin{array}{cc} 1 & 0 \\ 0 & -1 \end{array} \right) \\[3mm]
 & + \left( \frac{1}{\e \cosh^4 \sigma} \, L_{\sigma,\theta} \, \bar u + \frac{1}{\e^4 \cosh^4 \sigma} \, Q^2_{\sigma,\theta}(\bar u) \right)  \mathfrak g_0 ,
\end{align*}

\begin{align*}
 \mathfrak h_\e(\bar u)(\sigma,\theta) & = \mathfrak h_\e(\sigma,\theta) + \left( \begin{array}{cc} \partial_\sigma^2 \bar u & \partial_\sigma \partial_\theta 
 \bar u \\ 
 \partial_\sigma \partial_\theta \bar u & \partial_\theta^2 \bar u \end{array} \right) - \frac{\bar u}{\cosh^2 \sigma} \left( \begin{array}{cc} 1 & 0 \\ 0 & 
 1 \end{array} \right) \\[3mm] 
 & + \tanh \sigma \left( \begin{array}{cc} - \partial_\sigma \bar u & \partial_\theta \bar u \\ \partial_\theta \bar u & \partial_\sigma \bar u 
 \end{array} \right) + \left( \e \cosh \sigma \, L_{\sigma,\theta} \, \bar u + \frac{1}{\e \cosh^2 \sigma} \, Q^2_{\sigma,\theta}(\bar u) \right) \mathfrak 
 h_0(\sigma,\theta),
\end{align*}
where $\mathfrak h_0$ and $\mathfrak g_0$ are bounded symmetric 2-forms which do not depend on $\e$. Finally, going back to $ \bar w = \frac{1}{a} \, \bar u $, we get

\begin{equation} \label{MCcatNeck} 
 \begin{array}{ll} \mathcal H(\bar w) & =  \mathcal H(0) + \dfrac{1}{\e^2 \cosh^2 \sigma} \left( \partial_\sigma^2 + \partial_\theta^2 + 
 \dfrac{2}{\cosh^2 \sigma} \right) \bar w \\[5mm] 
 & + \dfrac{1}{\e \cosh \sigma} L_{\sigma,\theta} \, \bar w +  \dfrac{1}{\e^3 \cosh^4 \sigma} \, Q^2_{\sigma,\theta}(\bar w) + \dfrac{1}{\e^4 \cosh^4 \sigma} 
 \, Q^3_{\sigma,\theta}(\bar w). \end{array}
\end{equation}  

\medskip

\subsection{Mean curvature of the perturbed "gluing regions"}

Let $M$ be a smooth hypersurface in a smooth Riemannian manifold endowed with a metric $g$. Take $w$ a small smooth function and $V_1$ and $V_2$ two smooth vector fields on $M$. Let $\mathcal H^i(w)$ denote the mean curvature of the hypersurfaces obtained by perturbation of $M$ in the direction $ V_i$, $i= 1,2$. We have the following result:
\begin{lemma} \label{LMC} The following relation holds
$$ \left. D \mathcal H^2 \right|_{w=0}(v) = \left. D \mathcal H^1 \right|_{w=0} (\tau \, v) + g(\nabla_M \mathcal H(0), \mathcal T) $$
where $ \tau = \frac{|V_2^\bot|}{|V_1^\bot|} $, and $ \mathcal T = V_2^T - \tau \, V_1^T$, and where $V_i^\bot$ and $ V_i^T $ denote the orthogonal projections of $V_i$ on the normal and the tangent bundle of $M$.
\end{lemma}
\begin{proof}
This lemma is a simple generalisation of the result proven in \cite{Pac-Ros} where the case when one of the vector fields $V_i$ is a unit normal to $M$ is treated. The proof consists of applying the implicit function theorem to the equation
$$ 
p + t \, V_1(p) = q + s \, V_2(q), \quad p, q \in M, \quad t,s \in \mathbb R ,
$$
expressing locally $p$ and $t$ as functions of $q$ and $s$:
$$ 
p = \Phi(q,s) \quad \mbox{and} \quad t = \Psi(q,s) ,
$$
with $ \Phi(q,0) = q $ and $ \Psi(q,0) = 0$. We obtain then
$$ 
\partial_s \Psi (\cdot,0) [V_1]^\bot = [V_2]^\bot \quad \mbox{and} \quad \partial_s \Phi(\cdot,0) = [V_2]^T - \partial_s \Psi(\cdot,s) [V_1]^T .
$$
Moreover, we have
$$ 
\left. D \mathcal H^1 \right|_{w=0}(\partial_s \Psi(\cdot,0) \, v) + \nabla \mathcal H(0) \cdot \partial_s \Phi \, v = \left. D \mathcal H^2 \right|_{w=0}(v) ,
$$
and the result follows.
\end{proof}

\medskip

Now let us return to the surface $\tS$. Making use of the proof of the proposition 3.5, one can see that in the region $ \Omega_{glu}^0$
the components of $ \nabla^{\tilde g} \mathcal H $ are bounded by a constant times $\te^{1/2-\beta}$. Moreover, using the expression obtained in the lemma 3.1 for the normal vector field to the surface parametrized as a graph of the function $ \tilde \Upsilon $, we get
$$ 
[ \tilde \Xi_{\te} ]^N / [\partial_{x_3}]^N = 1 + \mathcal O(\te) \quad \mbox{and} \quad [ \tilde \Xi_{\te}  ]^T \sim [\tilde \partial_{x_3}]^T = \mathcal O(\te) .$$
Therefore, lemma \eqref{LMC} with $ V_1 = \partial_{x_3} $ and $ V_2 = \tilde \Xi_{\te} $ yields
$$ 
\mathcal H(w) = \mathcal H(0) + \Delta ( B \, w ) + L^{\gamma_0}_z \, w + {\te}^{-1} \, Q_z^{2,\gamma_0}(w) + {\te}^{-2} Q_z^{3,\gamma_0} .
$$
in $\Omega_{glu^0}$. Similarly, in $ \Omega_{glu}^m $ taking $ V_1 = \partial_{\xi_3} $ and $ V_2 = \bar \Xi_\e $, and using the fact that the components of the gradient of $\mathcal H$ are bounded by a constant times $\e^{-1/3-\beta}$ and the fact that 
$$ 
[ \bar \Xi_{\e} ]^N / [\partial_{\xi_3}]^N = 1 + \mathcal O(\e^{4/3}) \quad \mbox{and} \quad [ \bar \Xi_{\e}  ]^T \sim [\tilde \partial_{\xi_3}]^T = \mathcal O(\e^{4/3}) ,
$$ 
we get
\begin{multline*} 
 \mathcal H(w) = \mathcal H(0) + \Delta ( B \, w ) + \e^{-2/3 - \beta} \, L^{\gamma_m}_z \, w \, + \e^{-5/3 - \beta} \, Q_z^{2,\gamma_m}(w) + \e^{-8/3 - \beta} \, Q_z^{3,\gamma_m}(w),
\end{multline*}
for all positive $\beta \in (0,1)$.

\section{Linear analysis in the puncture disk}

We would like to analyse the Laplace operator subject to the Robin boundary data:
\begin{equation} \label{ProbLin}
 \left\{ 
\begin{array}{l}
\Delta \,  w =  f \quad \mbox{in} \quad D^2 \setminus \{0\} \quad (\mbox{or} \ D^2 ) \\[3mm]
\partial_r w - w = 0 \quad \mbox{on} \quad S^1 \setminus \{ z_1, \ldots, z_n \}
\end{array} \right.
\end{equation}
\newline where $f$ is a given function whose regularity and properties will be stated shortly. In what follows we suppose that we work in the domain $D^2 \setminus \{ 0 \}$. The case of the entire open disk $D^2$ can be treated in an analogous manner with certain simplifications. 

\medskip

First of all, we take $f$ even with respect to the angular variable and, for a given $n \geq 2 $, invariant under rotations by the angle $\frac{2 \pi}{n}$. With this assumption, the operator associated to (\ref{ProbLin}) does not have any bounded kernel and hence, the solvability of (\ref{ProbLin}) follows from classical arguments. For example, if $ f \in  {\mathcal C}^{0, \alpha} (\overline{ D^2})$ we get the existence of $ w \in {\mathcal C}^{2, \alpha} (\overline{D^2})$ solution $ w $ of (\ref{ProbLin}). Moreover,
$$ 
| w \|_{\mathcal C^{2,\alpha}(\overline{D^2})} \leq C \left( \|  w \|_{\mathcal C^0(\overline{D^2})} + \| f \|_{\mathcal C^{0,\alpha}(\overline{D^2})} \right) 
$$ 
We would like to understand what happens if we allow $f$ to have singularities at $0$ and/or $z_m$, $m = 1, \ldots, n$.

\medskip

We define the weighted spaces we will work in. As before we set
\[
\gamma (z) = |z| \, \overset{n}{\underset{m=1} \Pi} |z-z_m|,
\]
and we assume that we are given $ \nu  \in \mathbb R$. We say that a function $u \in L^{\infty}_{loc}(D^2) $ belongs to the space $L^{\infty}_\nu(D^2)$ if
$$ 
\| \gamma^{-\nu} u \|_{L^{\infty}(D^2)} < \infty .
$$
Let us use the notation $ D^2_* $ for the open punctured disc $ D^2 \setminus \{0\} $. The space $\mathcal{C}^{k,\alpha}_{\nu} (D^2_*)$ is defined to be the space of functions $u \in {\mathcal C}^{k ,\alpha}_{loc}(D^2_*)$ for which the following norm is finite
\[
\begin{array}{rllll}
\| u \|_{{\mathcal C}^{k ,\alpha}_{\nu} (D^2_*)} & := & \| \gamma^{-\nu} \, u \|_{\mathcal C^{k , \alpha} (D_*,  \gamma^{-2} \, g_{eucl})}.
\end{array}
\]

\medskip

Observe that, on the right hand side, we do not use the Euclidean metric to calculate the gradient of a function but rather a singular metric $ \gamma^{-2} \, g_{eucl}$. As a consequence, a function $u$ belongs to ${\mathcal C}^{k ,\alpha}_{\nu} (D^2_*)$ if 
\begin{multline*} 
\underset{ z \in D^2_* }{ \mathrm{sup} } \left| \gamma^{-\nu} \, u \right| + \sum_{i = 1}^k \, \underset{ z \in D^2_* }{ \mathrm{sup}} \left| \gamma^{-\nu + i} \, \nabla^i u \right| + \\
\underset{ z,z' \in D^2_* }{\mathrm{sup}} \left\{ \dfrac{ |\gamma^{- \nu + k + \alpha}(z) \, \nabla^k u(z) - \gamma^{- \nu + k + \alpha}(z') \nabla^k u(z')|}{|z-z'|^{\alpha}} \right\} < \infty
\end{multline*}

\medskip

Like in the section 4, instead of the problem \eqref{ProbLin}, we can consider an equivalent problem defined in $ \overline{D^2} \setminus \{0,1\}$. Take the change of variables $ \, z \, \mapsto \, z^n $ and notice that
$$ 
|z|^2 \, \Delta(z) = n^2 \, |z|^{2n} \, \Delta(z^n) .
$$

\smallskip

We take a function $ F $ in $ \, D^2 \setminus \{ 0 \}, \, $ such that 
$$ 
F(z^n) = \dfrac{1}{n^2} \, |z|^{2 - 2n} f(z) .
$$

\smallskip

Consider the problem:
\begin{equation} \label{ProbLin01}
 \left\{ 
\begin{array}{l}
\Delta \, W =  F \quad \mbox{in} \quad D^2 \setminus \{ 0 \} \\[3mm]
\partial_r W - \dfrac{1}{n} W = 0 \quad \mbox{on} \quad S^1 \setminus \{ 1 \}
\end{array} \right.
\end{equation}

\medskip

We define the space $ L^{\infty}_{\nu_0,\nu_1}(D^2) $ as the space of functions $ U \in L^{\infty}_{loc}(D^2) $ for which
$$ 
\| |z|^{-\nu_0}|z-1|^{-\nu_1} \, U \|_{L^{\infty}(D^2)} < \infty .
$$

Notice, if we take  $ f \in \, L^{\infty}_{\nu-2}(D^2) \, $, then $ \, F \in L^{\infty}_{\nu/n-2,\nu-2}(D^2) $ and
$$ 
\| F \|_{ L^{\infty}_{\nu/n-2,\nu-2}(D^2)} = \frac{1}{n^2} \, \| f \|_{L^{\infty}_{\nu-2}(D^2)} .
$$

\begin{proposition} \label{LinSol0}
Assume that $ \nu \in (0, 1)$. Then, there exists a constant $ C > 0 $ and, for all $ n \geq 2 $, for all $ F $, such that $ |z|^{-\nu/n + 2} \, F \in  L^{\infty}(D^2) $, there exist a unique function $ \Psi_0 $ and a unique constant $c_0^*$, such that $ W_0 := \Psi_0 + n \, c_0^* $ is a solution to (\ref{ProbLin01}) and 
$$ 
\| \, |z|^{-\nu/n} \, \Psi_0 \, \|_{L^{\infty}( D^2) } + |c_0^*| \leq C \, \| \, |z|^{ - \nu/n + 2} \, F \, \|_{L^{\infty}(D^2_* )} .
$$ 
\end{proposition}

\begin{proof}

First, let us assume that $F$ does not depend on the angular variable $ \phi $. In this case, (\ref{ProbLin01}) reduces to a second order ordinary differential equation which can be solved explicitly. 
\begin{multline*} 
 \Psi_0^{rad}(r) = \int_0^r \frac{1}{s} \int_0^s t \, F(t) \, dt \, ds, \quad W_0^{rad} = \Psi_0^{rad} + n \, c_0^* \\
 c_0^* = - \int_0^1 s \, F(s) \, ds + \frac{1}{n} \, \int_0^r \frac{1}{s} \int_0^s t \, F(t) \, dt \, ds
\end{multline*}
With little work, one checks that the result is indeed correct in this spacial case. 

\medskip

Furthermore, we claim that, if we restrict our attention to the space of functions for which 
\[
\int_{S^1} F(r e^{i\phi}) \, r \, d\phi =0
\]
for all $r \in (0,1)$, then there exists a function $W_0^{mean}$ such that 
$$ \| |z|^{-\nu/n} W_0^{mean} \|_{L^{\infty}(D^2)} \leq C \, \| |z|^{-\nu/n + 2} F \|_{L^{\infty}(D^2)} $$
for a constant $ C $ independent of $n$. We construct $W_0^{mean}$ as a limit of solutions to the Poisson's equation in annulus-type domains with mixed boundary date. 

\medskip

More precisely, take $\epsilon \in (0,1)$ and let us denote $A_\epsilon$ the annulus $D^2 \setminus D^2(\epsilon)$. For a fixed $n$ let $W_{\epsilon,n}$ be the solution to the problem
\begin{equation} \label{ProbEps}
 \left\{ \begin{array}{l}
 \Delta W_{\epsilon,n} = F \quad \mbox{in} \quad A_\epsilon, \\[3mm]
 \partial_r W_{\epsilon,n} - \frac{1}{n} W_{\epsilon,n} = 0 \quad \mbox{on} \quad S^1, \quad 
 W_{\epsilon,n} = 0 \quad \mbox{on} \quad S^1(\epsilon).
 \end{array} \right.
\end{equation}

There exists a constant $C(\epsilon,n)$ which depends on $\epsilon$ and $n$ and such that
$$ \| W_{\epsilon,n} \|_{L^\infty(A_\epsilon)} \leq C(\epsilon,n) \, \| F \|_{L^\infty(A_\epsilon)} $$
Changing the constant $C(\epsilon,n)$, we can rewrite this as follows
\begin{equation} \label{Contin}
 \| |z|^{-\nu/n} W_{\epsilon,n} \|_{L^\infty(A_\epsilon)} \leq C(\epsilon,n) \, \| |z|^{ - \nu/n + 2} F \|_{L^\infty(A_\epsilon)} 
\end{equation}
If the constant $ C(\epsilon,n) = C(n)$ didn't depend on $\epsilon$, then for every $\epsilon_0 \in (0,1)$, and for all $\epsilon < \epsilon_0$ we would have
$$ \| W_{\epsilon,n} \|_{L^\infty \left( A_{\e_0/2} \right)} \leq C(n) \, \| |z|^{ - \nu/n + 2} F \|_{L^\infty(D^2_*)}. $$
Then, by elliptic regularity theory, changing the constant $C(n)$ if necessary, we would have
$$ \| \nabla W_{\epsilon,n} \|_{L^\infty(A_{\epsilon_0})} \leq  C(n) \, \| |z|^{ - \nu/n + 2} F \|_{L^\infty(D^2_*)} $$
Thus, when $\epsilon$ tends to $0$, the sequence $W_{\epsilon,n}$ would admit a subsequence converging on compact sets of $D^2_*$ to a function $W_n$, a solution of \eqref{ProbLin01} for a fixed $n$, such that
$$ \| |z|^{-\nu/n} W_n \|_{L^\infty(D^2)} \leq  C(n) \, \| |z|^{ - \nu/n + 2} F \|_{L^\infty(D^2)} $$

The fact that the constant $ C(\e,n) $ doesn't depend on $ \epsilon $ can be proven by an argument by contradiction. We suppose, that there exists a sequence of parameters $ \epsilon_j $ and a sequence of points $z_j$ such that 
\begin{align*} 
 \| |z|^{-\nu/n} \, W_{j,n} \|_{L^\infty(A_j)} \leq 1, \quad & W_{j,n}(z_j) = |z_j|^{\nu/n}, \\
  \mbox{and} \quad & \Delta \, W_{j,n} = F_{j,n}, \quad \| |z|^{2-\nu/n} F_{j,n} \|_{L^{\infty}(A_j)} \underset{j \rightarrow \infty}{\rightarrow} 0
\end{align*}
where $\, W_{j,n} := \dfrac{1}{C(\epsilon_j,n)} \, W_{\epsilon_j,n}$, $\, F_{j,n} := \dfrac{1}{C(\epsilon_j,n)} \, F \,$ and $A_j = A_{\epsilon_j}$.

\medskip

We suppose first that the sequence $ z_j $ converges to a point $z_\infty \in D^2_* $.  We denote 
$$ 
\mathcal W_{j,n}(z) = W_{j,n} \left( |z_j| \, z \right) \, |z_j|^{-\nu/n} ,
$$
then, for every $j$, we have 
$$ 
\mathcal W_{j,n} \left( z_j / |z_j| \right) = 1. 
$$ 
The sequence $\mathcal W_{j,n}$ admits a subsequence converging on compact sets to a function $\mathcal W_n$ which is a solution to
$$ 
\left\{ \begin{array}{l} 
    \Delta \, \mathcal W_n = 0 \quad \mbox{in} \quad D^2_* \\[3mm]
    \partial_r \mathcal W_n - \frac{1}{n} \mathcal W_n = 0 \quad \mbox{on} \quad S^1
   \end{array} \right. .
$$
Moreover, we have $ \left| \mathcal W_n(z) \right| \leq |z|^{\nu/n} $ and $ \mathcal W_n \left( \frac{z_\infty}{|z_\infty|} \right) = 1 $. Using the fact that $\mathcal W_n$ has no radial part and that the problem \eqref{ProbLin01} has no bounded kernel, we get a contradiction.

\medskip

When the sequence of points $z_j$ tends to $0$ at the same time as $ \frac{|z_j|}{\epsilon_j} \underset{j \rightarrow \infty}{\rightarrow} 0 $ we obtain a sequence of functions $\mathcal W_{j,n}$, which admits a subsequence converging on compact sets to a function $\mathcal W_n$ which is a solution to the problem
$$ \Delta \, \mathcal W_n = 0, \quad \mbox{in} \quad \mathbb R^2 \setminus \{0\}, \quad \left| \mathcal W_n \right| \leq c \, |z|^{\nu/n}, $$
which implies $W_n \equiv 0$ and contradicts the fact that $\mathcal W_{j,n}\left( \frac{z_j}{|z_j|} \right) = 1$ for all $j$. 

\medskip

It remains to deal with the case when $z_j \underset{j \rightarrow \infty}{\rightarrow} 0$ and $ \frac{|z_j|}{\epsilon_j} \underset{j \rightarrow \infty}{\rightarrow} a $, where $a$ is a constant strictly greater than $1$. In this case $\mathcal W_{n,j}$ admits a subsequence converging on compact sets to a function $\mathcal W_n$, which is a solution to
$$ 
\left\{ \begin{array}{l} 
    \Delta \, \mathcal W_n = 0 \quad \mbox{in} \quad \mathbb R^2 \setminus D^2(a) \\[3mm]
    \mathcal W_n = 0 \quad \mbox{on} \quad S^1(a)
   \end{array} \right. 
$$
and such that $\left| \mathcal W_n \right| \leq c \, |z|^{\nu}$. Once again, this implies $\mathcal W_n \equiv 0$ and gives a contradiction. 

\medskip

Finally, the case when $z_j \underset{j \rightarrow \infty}{\rightarrow} 0$ and $ \frac{|z_j|}{\epsilon_j} \underset{j \rightarrow \infty}{\rightarrow} 1 $ doesn't happen. For every $j$ we have
$$ \left\{ \begin{array}{l} 
    \Delta \, W_{j,n} = F_{j,n} \quad \mbox{in} \quad \mathbb D^2( 2 \epsilon_j) \setminus D^2 \left( \epsilon_j \right) \\[3mm]
    W_{j,n} = 0 \quad \mbox{on} \quad S^1 \left( \epsilon_j \right)
   \end{array} \right. $$
Moreover, 
$$ \, \left| F_{j,n} \right| \leq \epsilon_j^{\nu/n-2} \quad \mbox{and} \quad  \left| W_{j,n} \right| \leq \epsilon_j^{\nu/n}. $$ 
Then in the subsets of $ \mathbb D^2( 2 \epsilon_j) \setminus D^2 \left( \epsilon_j \right) $ we have $ \left| \nabla W_{j,n} \right| \leq c \, \epsilon_j^{\nu/n-1} $. This implies that in the neighbourhood of $|z| = \epsilon_j$, we have
$$ \left| W_{j,n} \right| \leq C \, \epsilon_j^{\nu/n - 1} \left( |z| - \epsilon_j \right). $$
At $ z = z_j $ this yields $ \frac{|z_j|}{\epsilon_j} - 1 \geq C $,
which is not possible starting from a certain $j$.

\medskip

Similarly, we can prove that the constant $ C(n) $ in \eqref{Contin} does not depend on $n$. If it were not the case we could define a sequence of function $\tilde W_n$ and a sequence of points $z_n$, such that $\tilde W_n \left( z_n/|z_n| \right) = 1$. Then $\tilde W_n$ would admit a subsequence converging on compact sets to a function $\tilde W$, which is harmonic in a unit disk and has homogeneous Neumann boundary data. Using that $\int_{D^2} \tilde W \, dx_1 \, dx_2 = 0$, we get the contradiction.

\end{proof}

Let us fix a cut-off function $\chi$ defined in the unit disk $D^2$ which is identically equal to $1$ in a neighbourhood of $z=1$ and to $0$ in a neighbourhood of $z=0$. We define deficiency spaces
$$ \mathfrak D_n = \mathrm{span} \{ n \} \quad \mbox{and} \quad \mathfrak D_{\chi} = \mathrm{span} \{ \chi \} $$

\begin{proposition} \label{LinSol01}
 Assume that $ \nu \in (0, 1)$. Then, there exists a constant $ C > 0 $ and, for all $ n \geq 2 $, for all $ \, F \in L^{\infty}_{\nu/n -  
 2, \nu-2}(D^2) \, $ there exist a unique function $ \Psi \in L^\infty_{\nu/n, \nu}(D^2) $ and unique constants $c_0^*$ and $c_1^*$, such that $ 
 \, W := \Psi + n \, c_0^* + c_1^* \, \chi \, $ is a solution to (\ref{ProbLin01}) and such that
 $$ \| W \|_{L^{\infty}_{\nu/n,\nu}( D^2) \oplus \mathfrak D_n \oplus \mathfrak D_{\chi} } < C \, \|  F \|_{L^{\infty}_{\nu/n-2,\nu-2}(D^2)} $$ 
\end{proposition}

\begin{proof}

We take the conformal mapping
$$ \lambda: \, \mathbb C_- \longrightarrow D^2, \quad  \lambda(\zeta) = \frac{1+\zeta}{1-\zeta}. $$
which sends a half-disk in $\mathbb C_-$ centered at $0$ and of radius $ \rho \in (0,1) $ to the intersection of the unit disk $ D^2 $ with the disk of radius $ r_\rho = \frac{2\rho}{1-\rho^2}$ centered at $ c_\rho = 1 + \frac{2 \rho^2}{1 - \rho^2}$. For example, for $ \rho = \frac{1}{3} $, we get $ r_{\frac{1}{3}} = \frac{3}{4} \quad \mbox{and} \quad c_{\frac{1}{3}} = \frac{5}{4} $ and for $ a = \frac{1}{5} $, we  get $ r_{\frac{1}{5}} = \frac{5}{12} \quad \mbox{and} \quad c_{\frac{1}{5}} = \frac{13}{12}. $

\medskip

We define a cut-off function $ \bar \chi \in \mathcal C^{\infty}(\mathbb R^2)$, such that 
$$ \bar \chi(\zeta) = \bar \chi(|\zeta|), \quad \bar \chi \equiv 0 \quad \mbox{for} \quad |\zeta| \geq 1/3 \quad \mbox{and} \quad \bar \chi \equiv 1 \quad \mbox{for} \quad |\zeta| \leq 1/5 $$ 
and put $ \chi(z) = \bar \chi \left(| \lambda^{-1}(z) | \right) $. Then, we have $ \left. \partial_r \chi \right|_{r=1} = 0 $ and
$$ \chi(z) \equiv 0 \quad \mbox{for} \quad |z - 5/4| \geq 3/4 \quad \mbox{and} \quad \chi \equiv 1 \quad \mbox{for} \quad |z - 13/12| \leq 5/12, $$
 
We decompose
$$ 
F(z) = F_0(z) + F_1(z) = (1 - \chi(z)) \, F(z) + \chi(z) \, F(z). 
$$
Then, we have
\begin{multline*} 
 \| \, |z|^{ - \nu/n + 2} F_0 \, \|_{L^\infty(D^2)} \leq \| F \|_{L^\infty_{\nu/n-2,\nu-2}(D^2)}, \\[3mm] 
 \| \, |z-1|^{ - \nu + 2} F_1 \, \|_{L^\infty(D^2)} \leq \| F \|_{L^\infty_{\nu/n-2,\nu-2}(D^2)} 
\end{multline*} 

We define $ \overline F(\zeta) = F(\lambda(\zeta)) $. Remark that
$$ 
\Delta_z = \dfrac{ |1-\zeta|^2 }{4} \, \Delta_\zeta ,
$$
and consider the problem
$$ 
\left\{ \begin{array}{l} 
 \Delta \overline W_1 = \frac{4}{|1-\zeta|^2} \, \overline F_1(\zeta) \quad \mbox{in} \quad \mathbb C_- \cap D^2_*(1/3) \\[3mm]
 \partial_{\xi_1} \, \overline W_1 = 0 \quad \mbox{on} \quad \partial \mathbb C_- \cap D^2_*(1/3), \quad  \\[3mm]
 \overline W_1 = 0 \quad \mbox{on} \quad \mathbb C_- \cap \partial D^2_*(1/3)  \end{array} \right. .
 $$
We extend $\overline F_1$ by symmetry to $ D^2_*(1/3) $ and consider the problem
\begin{equation} \label{ProbLin1} 
 \left\{ \begin{array}{l} \Delta \overline W_1 = \hat F_1 \quad \mbox{in} \quad D^2_*(1/3) \\[3mm]
 \overline W_1 = 0 \quad \mbox{on} \quad S^1(1/3) \end{array} \right.
\end{equation} 
where $\hat F_1 = \frac{4}{|1-\zeta|^2} \, \overline F_1(\zeta)$. Automatically, the restriction of $ \overline W_1 $ to $ \, \mathbb C_- \cap D^2_*(1/3) \,$ satisfies  $ \, \partial_{\xi_1} \overline W_1 = 0 $ at $ \xi_1 = 0. $

\medskip

The existence and the properties of $ \overline W_1 $ are obtained in the same way as the existence and the properties of $W_0$ in the proposition \eqref{LinSol0}. We suppose first that the function $\hat F_1$ doesn't depend on the angular variable and depends only on $ |\zeta| = \rho $. Then, the function
\begin{multline*} 
 \overline \Psi^{rad}_1(\rho) = \int_0^\rho \frac{1}{s} \int_0^s t \, \hat F_1(t) \, dt \, ds, \quad \overline W_1^{rad} = \overline \Psi_1^{rad}(\rho) + 
  c_1^*, \\
  c_1^* = - \int_0^{1/3} \frac{1}{s} \int_0^s t \, \hat F_1(t) \, dt \, ds
\end{multline*}
satisfies \eqref{ProbLin1} and using that $ \, \| \, |\zeta|^{-\nu + 2} \, \hat F \, \|_{L^\infty(D^2(1/3))} \leq \| F \|_{L^\infty_{\nu/n-2, \nu - 2}(D^2)}, \, $ we get
$$ 
\| \, |\zeta|^{- \nu} \, \overline \Psi_1 \, \|_{L^\infty(D^2)} + |c_1^*| \leq C \, \| F \|_{L^\infty_{\nu/n-2, \nu - 2}(D^2)} .
$$
On the other hand, if 
$$ 
\int_{S^1} \hat F_1(\rho,\theta) \, \rho \, d\theta = 0, \quad \mbox{for all} \quad \rho \in (0,1) ,
$$ 
using the same argument as in the previous proposition, one finds a function $\overline W_1^{mean}$, which satisfies \eqref{ProbLin1} and such that
$$ 
\| |\zeta|^{-\nu} \, \overline W_1^{mean} \|_{L^{\infty}(D^2(1/3))} \leq C \, \| F \|_{L^\infty_{\nu/n-2, \nu - 2}(D^2)} $$
Finally, we put 
$$ 
\, \overline W_1 : = \overline W_1^{mean} + W_1^{rad},  \quad \mbox{and} \quad W_1 : = \overline W_1 \circ \lambda^{-1} .
$$ 
The function $ \chi \, W_1$ is defined in a neighbourhood of $ z = 1 $ and can be extended by zero to the entire punctured unit disc $ D^2_* $. We have
$$ 
\left\{ \begin{array}{l} \Delta(\chi \, W_1) = F_1 + 2 \nabla \chi \, \nabla W_1 + W_1 \, \Delta \, \chi \quad \mbox{in} \quad D^2_* \\[3mm]
   \partial_r (\chi \, W_1) = 0 \quad \mbox{on} \quad S^1 \setminus \{ 1 \} \end{array} \right. $$

The function $ \nabla \chi \, \nabla W_1 + W_1 \, \Delta \, \chi $ belongs to $ L^\infty_{\nu/n-2, \nu - 2}(D^2) $ and has compact support, since is identically zero in the neighbourhood of $ z = 0 $ and $ z = 1 $. According to the proposition \eqref{LinSol0} we can find a function $ W_0 $ which satisfies
$$ 
\left\{ \begin{array}{l} \Delta \, W_0 = F_0 - 2 \nabla \chi \, \nabla W_1 - W_1 \, \Delta \, \chi \quad \mbox{in} \quad D^2_* \\[3mm]
   \partial_r W_0 - \frac{1}{n} W_0 = 0 \quad \mbox{on} \quad S^1 \setminus \{ 1 \} \end{array} \right. $$
By the elliptic regularity in weighted spaces we have
\begin{multline*} 
 \| \, |z-1|^{- \nu} \, W_1 \, \|_{L^\infty(D^2)} \leq C \| \, \| z-1 \|^{ - \nu + 2} \, F_1 \|_{L^\infty(D^2)}, \\[3mm] 
 \| \, |z-1|^{-\nu + 1 } \, \nabla W_1 \, \|_{L^\infty(D^2)} \leq C \| \, \| z-1 \|^{ - \nu + 2} \, F_1 \|_{L^\infty(D^2)}
\end{multline*}
Then, 
$$ 
\| \, |z|^{ - \nu/n + 2} \, \left( F_0 - 2 \nabla \chi \, \nabla W_1 - W_1 \, \Delta \, \chi \right) \, \|_{L^\infty(D^2)} \leq C \|  F \|_{L^\infty_{\nu/n - 2, \nu - 2}(D^2)} .
$$ 
So, we can write $ \, W_0 = \Psi_0 + c_0^* \, n $, where
$$ 
\| \, |z|^{- \nu/n} \, \Psi_0 \, \|_{L^\infty(D^2)} + |c_0^*| \leq C \| F \|_{L^\infty_{\nu/n - 2, \nu - 2}(D^2)} .
$$
The function
$$ 
W_{almost} : =  W_0 + \chi \, W_1 ,
$$
satisfies the problem
$$ 
\left\{ \begin{array}{l} 
 \Delta \, W_{almost} = F \quad \mbox{in} \quad D^2_* \\[3mm]
 \partial_r \, W_{almost} - \frac{1}{n} W_{almost} = - \frac{1}{n} \, \chi \, W_1 \quad \mbox{on} \quad S^1 \setminus \{ 1 \}  \end{array} \right. 
 $$
Take the function
$$ 
h(z) : = \frac{|z|^2 - 1}{2n} \, \chi \, W_1(z). 
$$
then, 
$$ 
\partial_r h - \frac{1}{n} h = \frac{1}{n} \chi \,W_1 \quad \mbox{at} \quad r = 1 \quad \mbox{and} \quad
\Delta h = \frac{|z|^2-1}{2n} \, \chi \, F_1 + \frac{2 r}{n} \partial_r \, ( \chi \, W_1 ) + \frac{2}{n} \, \chi \,  W_1 .
$$ 
Consider the operator 
$$ 
G_h : \, L^{\infty}_{\nu/n-2,\nu-2} \rightarrow L^{\infty}_{\nu/n,\nu} \oplus \mathfrak D_n \oplus \mathfrak D_\chi, \quad  G_h(f) = W_h: =  W_{almost} + h. 
$$
We have 
 $$ \Delta \circ G_h = Id + R_h, \quad R_h \, : L^\infty_{\nu/n-2, \nu-2} \longrightarrow L^\infty_{\nu/n-2, \nu-2} $$
 $$ R_h(f) = \Delta \, h, \quad  \| R_h \| \leq \frac{1}{n} $$
Finally, we define a continuous linear operator $ G = G_h \circ \left( Id + R_h \right)^{-1} $ and the function $ W = G(f) $, the unique solution to \eqref{ProbLin01} which can be written in the form
$$ 
W = \Psi + n \, c_0^* + c_1^* \, \chi \,  \in  L^{\infty}_{\nu/n,\nu}(D^2) \oplus \mathfrak D_n \oplus \mathfrak D_\chi.
$$

\end{proof}

Now we can go back to the initial problem \eqref{ProbLin}. Take $ f(z) = n^2 \, |z|^{2n-2} F(z^n) $ and put $w(z) = W(z^n)$. Then, $ w \in L^{\infty}_{\nu}(D^2) \oplus \mathfrak D_n \oplus \mathfrak D_{\chi_n} $ and can be written in the form
$$ w = \psi + n \, c_0^* + c_1^* \, \chi_n(z), \quad \chi_n(z) = \chi(z^n), \quad \| \gamma^{-\nu} \, \psi \|_{L^{\infty}(D^2)} \leq C \, \| \gamma^{-\nu + 2} f \|_{L^{\infty}(D^2)}. $$

Finally, if we take $ \, f \in \mathcal C^{0,\alpha}_{\nu-2}(D^2_*) \,$, then by classical arguments of the elliptic theory in H\"older weighted spaces $ \, \psi \in  \mathcal C^{0,\alpha}_\nu(D^2_*) \, $ and there exists a constant $C$ such that
$$ 
\|  \psi \|_{C^{2,\alpha}_{\nu}(D^2_*)} \leq \, C \, \| f \|_{C^{0,\alpha}_{\nu-2}(D^2_*)}. 
$$

\section{linear analysis around the catenoidal bridges}

To analyse the linearised mean curvature operator in the neighbourhood of the catenoidal bridges we consider the following problem

\begin{equation} \label{LinCat}
 \left\{ \begin{array}{l} \bar L_{cat} \, w  = f  \quad \mbox{in} \quad \Cyl \\[3mm]
 \partial_\theta w = 0 \quad \mbox{on} \quad \mathbb R \times \{ \frac{\pi}{2}, \frac{3\pi}{2} \}
\end{array} \right.
\end{equation}
\newline where $ \, \bar L_{cat} = \partial_\sigma^2 + \partial_\theta^2 + \frac{2}{\cosh^2 \sigma}, \quad (\sigma, \theta) \in \mathbb R \times \left[ \frac{\pi}{2}, \frac{3 \pi}{2} \right]. $

\medskip

\begin{lemma} \label{ker}
Assume that $\delta\in (-1, 0) \cup (0,1)$. The subspace of $ (\cosh \sigma)^\delta \mathcal{C}^{2,\alpha} \left( \Cyl \right)$ that is invariant by $(\sigma,\theta)\ \mapsto\ (\sigma,-\theta)$ and $(\sigma,\theta) \ \mapsto \ (-\sigma, \theta)$ and solves
\begin{equation*}
\left\{ \begin{array}{l}
 \bar L_{cat} \, w = 0 \quad \mbox{in} \quad \Cyl \\[3mm]
 \partial_\theta w = 0 \quad \mbox{on} \quad \mathbb R \times \{ \frac{\pi}{2},\frac{3\pi}{2} \}
\end{array} \right.
\end{equation*}
is trivial when $ \delta \in (-1,0) $ and is one dimensional and spanned by $ \sigma \tanh \sigma - 1 $ when $\delta \in (0,1)$.
\end{lemma}

\begin{proof} We decompose $w$ in Fourier series
$$
w(\sigma,\theta)=\sum_{j \in \mathbb Z} w_j(\sigma) e^{i j \theta}.
$$
then the functions $w_j$ are solutions of the ordinary equation 
\[
\left(\partial^2_\sigma - j^2 + \dfrac{2}{\cosh^2 \sigma} \right) w_{j} = 0.
\] 
These solutions are asymptotic either to $(\cosh \sigma)^{j}$ or to $(\cosh \sigma)^{-j}$. By hypothesis, the solution is bounded by a constant times $(\cosh \sigma)^{\delta}$ and $ |\delta| < 1 $, so the solution has to be asymptotic to $(\cosh \sigma)^{-j}$, and then the solution is bounded. On the other hand, $-(j)^2 + \dfrac{2}{\cosh^2 \sigma} \leq 0$, so the maximum principle assures that $ w_j=0 $, for all $ j \geq 2 $. 

\medskip

Observe that the imposed symmetry $ (\sigma,\theta) \, \mapsto (\sigma,-\theta) $ and the boundary condition imply $ w_1 = 0 $. When $ j=0 $, $w_0$ is the solution of the ordinary equation
\[
\left(\partial^2_\sigma + \dfrac{2}{\cosh^2 \sigma} \right) w_{0}=0.
\]
By direct computations, we can see that $\,\tanh \sigma \,$ and $ \, \sigma \tanh \sigma - 1 \,$ are two independent solutions. The only solution symmetric with respect to the horizontal plane is $\sigma \tanh \sigma - 1$ and it belongs $ (\cosh \sigma)^\delta \mathcal{C}^{2,\alpha}(\Cyl)$ only when $ \delta \in (0,1) $.

\end{proof}

The next step is to prove that, under some hypothesis, there exists a right inverse of the problem \eqref{LinCat} and it is bounded. 

\begin{proposition} \label{LinSolCat}
Assume that $\delta \in (-1,0) \cup (0,1)$. Then given $ f \in (\cosh \sigma)^\delta \mathcal{C}(\Cyl)$, such that $ f(\sigma,\theta) = f(-\sigma,\theta) = f(\sigma,-\theta) $
there exists a unique constant $d_1^*$ and a unique function $ v \in (\cosh \sigma)^\delta \mathcal{C}^{2,\alpha}(\Cyl)$ such that the function $ w = v + d_1^*$ solves
\begin{equation} \label{syscat} 
\left\{\begin{array}{l}
\left(\partial^2_\sigma + \partial^2_\theta + \dfrac{2}{\cosh^2 \sigma} \right) w = f, \quad \mbox{in} \quad \Cyl \\[3mm]
 \partial_\theta w  = 0,  \quad \mbox{on} \quad \mathbb R \times \{ \frac{\pi}{2}, \frac{3\pi}{2} \}
\end{array} 
\right. 
\end{equation}
\newline and $ \, w(\sigma,\theta) = w(-\sigma,\theta) = w(\sigma,-\theta) \, $. Moreover, we have
\begin{equation} \label{estcat}
 \| (\cosh \sigma)^{- \delta} w \|_{\mathcal C^{2,\alpha} \left( \Cyl \right) } + | d_1^* | \leq C \, \| (\cosh \sigma)^{-\delta} f \|_{\mathcal C^{0,\alpha} 
 \left( \Cyl \right)}
\end{equation} 
\end{proposition}

\begin{proof}
Let us extand the function $f$ by symmetry to the entire unit cylinder $\mathbb R \times S^1$. Then, there exists a function $w$, which satisfies
\begin{align*} \label{eqcat}
 & \left(\partial^2_\sigma + \partial^2_\theta + \dfrac{2}{\cosh^2 \sigma} \right) w = f \quad \mbox{in} \quad \mathbb R \times S^1, \\[3mm]
 & \mbox{and} \quad w = v + d_1^*, \quad \| (\cosh \sigma)^{-\delta} v \|_{\mathcal C^{2,\alpha}(\mathbb R \times S^1)} + |d_1^*| \leq C \, \| (\cosh \sigma)^{-\delta} f \|_{\mathcal C^{0,\alpha}(\mathbb R \times S^1)}
\end{align*}
\newline This fact follows from the construction given by R. Mazzeo, F. Pacard and D. Pollack in \cite{Maz-Pac-Pol}. Here, we give a short sketch of their proof for the sake of completeness. Let first $f$ be a function whose Fourier series in $\theta$ is given by
$$ 
f(\sigma, \theta) = \sum_{|j| > 2} f_j(\sigma) \, e^{ij \theta} .
$$
Then, for every $ t \in \mathbb R $, there exists a function $ v_t = \sum_{|j|>2} v_j^t(\sigma) \, e^{i j \theta} $, a unique solution of the problem
$$ 
\left( \dfrac{d^2}{ds^2} - j^2 + \dfrac{2}{\cosh^2 \sigma} \right) \, v_j^t = f_j \quad \mbox{in} \quad |\sigma| < t, \quad v_j^t(\pm t) = 0, \quad j \geq 2 .
$$
One can prove this using the maximum principal and the method of sub- and supersolutions, taking $ \dfrac{1}{j^2 - 2 -\delta} \, (\cosh \sigma)^\delta $ as a barrier function. Taking a sum over $ |j|>2 $, we get a function $ v_t $ such that $ \bar L_{cat} v_t = f $ and $ v_t(\pm t) = 0 $. By the Schauder's elliptic theory, there exists a constant $C$ such that
$$ 
\| (\cosh \sigma)^{-\delta} v \|_{\mathcal C^{2,\alpha}( (-t,t) \times S^1 )} \leq C \| (\cosh \sigma)^{-\delta} f \|_{\mathcal C^{0,\alpha}( (-t,t) \times S^1 )}. 
$$
Moreover the constant $C$ does not depend on $t$, which can be proven by contradiction, using the same argument in the proposition \eqref{LinSol0}. Finally, the sequence $ v_t $ admits a subsequence which converges to a function $ v $ on compact subsets of $ \Cyl $ as $t$ tends to infinity and such that \eqref{estcat} is true.

\medskip

In the case when $ f = f_0(\sigma) + f_{\pm 1}(\sigma) \, e^{ \pm i\theta} $, we can construct a solution explicitly, taking 
$$ w_{\pm 1}(\sigma) = \cosh^{-1} \sigma \int_0^\sigma \cosh^2 t \int_0^t \cosh^{-1} \xi \, f_{\pm 1}(\xi) \, d \xi \, dt. $$
and
$$ w_0(\sigma) = \tanh \sigma \int_0^\sigma \tanh^{-2} t \int_0^t \tanh \xi \, f_0(\xi) \, d \xi \, dt $$
\newline Remark, that for $ |f_j(\sigma)| \leq (\cosh \sigma)^{\delta} $ for $j = 0, \pm 1$ there exist constants $ d $ and $ d_1^* $, such that 
$$ w_0 + d \, (1 - s \tanh s) = v_0 + d_1^*, \quad | v_0 | \leq c \, (\cosh \sigma)^{\delta} $$
moreover
$$ | w_{\pm 1} | \leq c \, (\cosh \sigma)^{\delta} $$ 
The estimates for derivatives of $w_0$ and $w_{\pm 1}$ are obtained by Schauder's theory. For all $ \delta \in (-1,1) $ we have
$$ \| (\cosh \sigma)^{-\delta} ( v_0 + w_{\pm 1} ) \|_{\mathcal C^{2,\alpha}\left( \mathbb R \times S^1 \right) } + |d_1^*| \leq \| (\cosh \sigma)^{-\delta} f \|_{\mathcal C^{0,\alpha} \left( \mathbb R \times S^1 \right)} $$

\medskip

Finally, by symmetry the restriction of $w$ to $ \mathbb R \times \left[ \frac{\pi}{2}, \frac{3\pi}{2} \right] $ satisfies $ \left. \partial_\theta \, w \right|_{ \{ \frac{\pi}{2},\frac{3\pi}{2} \} } = 0. $

\end{proof}
\section{Linear analysis around the catenoidal neck}

In this section in order study the linearised mean curvature operator around the catenoidal neck we consider the equation
\begin{equation} \label{LinCat}
 L_{cat} \, w  = f  \quad \mbox{in} \quad \mathbb R \times S^1 ,
\end{equation}
where $ \, L_{cat} = \partial_s^2 + \partial_\phi^2 + \frac{2}{\cosh^2 s} $. 

\medskip

We restrict our attention to functions which are even in the variables $\phi$ and $s$ and invariant under rotations by the angle $\frac{2\pi}{n}$. 
Given $ f \in (\cosh s)^{\delta} \mathcal C^{0,\alpha}(\mathbb R \times S^1), \, $ such that 
$$ 
f(s,\phi) = f(-s,\phi)= f(s, -\phi) = f(s, \phi + \pi/n). 
$$ 
we define $ \, F(s,\phi) = \frac{1}{n^2} f(\frac{s}{n}, \frac{\phi}{n}) \, $ and consider the problem
\begin{equation} \label{LinCatN}
L_{cat}^n \, W = \left( \partial_s^2 - j^2 + \frac{2}{n^2 \cosh^2 \frac{s}{n}} \right) W = F.
\end{equation}

We prove the following two lemmas:

\begin{lemma} \label{ker0}
Assume that $ \delta \in (-1,0) \cup (0,1)$. The subspace of $ \,(\cosh \frac{s}{n})^{\delta} \, L^\infty \left( \mathbb R \times S^1 \right) $ which is invariant by $ (s,\phi) \ \mapsto \ (s,-\phi)$ and $(s,\phi) \ \mapsto \ (-s, \phi) $ and solves
$$ 
L_{cat}^n \,  W = 0 \quad \mbox{in} \quad \mathbb R \times S^1 ,
$$
is trivial for $ \delta \in (-1,0) $ and is one dimensional and spanned by $ \frac{s}{n} \tanh \frac{s}{n} - 1 $ for $ \delta \in (0,1) $.
\end{lemma}

\begin{proof}
 The proof of this lemma is analogous to the proof of the lemma \eqref{ker} and uses the maximum principal for the Fourier modes $ j \geq 1 $ and the symmetry 
 with respect to the horizontal plain for $j=0$.  
\end{proof}

\begin{proposition} \label{LinProbCat0}
Assume that $ \delta \in (-1,0) \cup (0,1)$. Then, given a function \newline $ 
F \in \, \left( \cosh \frac{s}{n} \right)^{  \delta } L^\infty \left( \mathbb R \times S^1 \right), $ such that $ F(s,\phi) = F(-s,\phi) = F(s,-\phi) ,
$ there exist a unique constant $ d_0^* $ and a unique function $ V \in \left( \cosh \frac{s}{n} \right)^{ \delta} \, L^\infty(\mathbb R \times S^1) $ such that the function $ \, W = V + d_0^* \, $ solves
\begin{equation} \label{eqcat0} 
\left( \partial_s^2 + \partial_\phi^2 + \frac{2}{n^2 \cosh^2 \frac{s}{n}} \right) W = F ,
\end{equation}
and there exists a constant $C$, which does not depend on $n$, such that
\begin{equation} \label{estcat0} 
 \| \left( \cosh \frac{s}{n} \right)^{-\delta}  W \|_{ L^\infty (\mathbb R \times S^1)} + |d_0^*| \leq C \, \| \left( \cosh \frac{s}{n} 
 \right)^{-\delta} F \|_{ L^\infty(\mathbb R \times S^1)} .
\end{equation}
\end{proposition}

\medskip

\begin{proof}

We decompose both $F$ and $W$ in Fourier series
$$ 
F = \sum_{j \in \mathbb Z} F_j(s) \, e^{i j \phi}, \quad \mbox{and} \quad  W = \sum_{j \in \mathbb Z} W_j(s) \, e^{i j \phi}. 
$$
First, let $ F(s,\phi) = \sum\limits_{|j|>1} F_j(s) \, e^{i \phi j} $. Then, for every $ t \in \mathbb R $, using the method introduced in \cite{Maz-Pac-Pol}, we can solve 
$$ 
\left( \partial_s^2 - j^2 + \frac{2}{n^2 \cosh^2 \frac{s}{n}} \right) V_j^t = F_j, \quad V_j^t( \pm t ) = 0 ,
$$
by the maximum principal taking $ \dfrac{1}{j^2 n^2 - 2 - \delta} \, (\cosh \frac{s}{n})^\delta $ as a barrier function. When $t$ tends to infinity, we get a sequence of functions which admits a subsequence converging on compact sets of $ \mathbb R \times S^1 $ which satisfies \eqref{eqcat0} and \eqref{estcat0}.
When $ F(s,\phi) = F_0(s) $ we find explicitly
$$ 
W_0(s) =  \tanh \frac{s}{n} \int_0^\frac{s}{n} \tanh^{-2}t \int_0^{t} \tanh \xi \, F_0(n \xi) \, d \xi \, dt. 
$$
Like in the proposition 9.1 there exist a function $ V_0 \in \left( \cosh \frac{s}{n} \right)^{\delta} L^\infty(\mathbb R \times S^1) $ and a constant $ d_0^* $, such that the function $ W_0 = V_0 + d_0^* $ satisfies \eqref{eqcat0} and \eqref{estcat0}.

\end{proof}

Remark now that the function $ v(s, \phi) = V(ns,n\phi) $ is invariant under rotations by the angle $\frac{2\pi}{n}$ and satisfies
$$ 
L_{cat} \, v = f, \quad \mbox{and} \quad \| (\cosh s)^{-\delta} v \|_{L^\infty(\mathbb R \times S^1)} \leq C \, \| (\cosh s)^{-\delta} f \|_{L^\infty(\mathbb R \times S^1)} .
$$

Finally, by the Schauder's theory, if $ f \in (\cosh s)^\delta \mathcal C^{0,\alpha}(\mathbb R \times S^1) $, then $ v \in (\cosh s)^\delta \mathcal C^{2,\alpha}(\mathbb R \times S^1) $ and
$$ \| (\cosh s)^{-\delta} v \|_{\mathcal C^{2,\alpha}(\mathbb R \times S^1)} \leq C \, \| (\cosh s)^{-\delta} f \|_{\mathcal C^{0,\alpha}(\mathbb R \times S^1)} $$

\section{ The Fixed Point Theorem argument }

We parametrize $\tS$ by the following sub-domain of the unit disc:
$$ 
\Omega_\e = \ \{ z \in D^2 \, : \, \te  \leq |z| \leq 1 \} \, \setminus \, \overset{n}{\underset{m=1}{\cup}} \lambda_m \{ \zeta \in \mathbb C_- \, : |\zeta| \leq \frac{\e}{2} \}. 
$$
Take a real number $ \nu \in (0,1) $. We denote $ \mathcal E^{k,\alpha}_{\nu,n} $ the Banach space which is a subspace of $ \mathcal C^{k,\alpha}_\nu(\Omega_\e) $ invariant under the transformation $z \mapsto \bar z$ and the rotations by the angle $\frac{\pi}{n}$. Remark, that when we use the change of variables
$$ 
z = \te \cosh s \, e^{i \phi}  \quad \mbox{or} \quad z = \lambda_m( 1/2 \, \e \cosh \sigma \, e^{i \theta} ) ,
$$
the functions 
$$ 
(s, \phi) \, \mapsto \, w(\te \cosh s \, e^{i \phi}) \quad \mbox{and} \quad (\sigma, \theta) \, \mapsto \, w( \lambda_m ( 1/2 \, \e \cosh \sigma \, e^{i \theta} )) ,
$$
belong to the functional spaces 
$$ (\te \cosh s)^{\nu} \mathcal C^{2,\alpha}( (-s_*,s_*) \times S^1 ) \quad \mbox{and} \quad (\e \cosh \sigma)^{\nu} \mathcal C^{2,\alpha}( (-\sigma_*,\sigma_*) \times \left[ \pi/2, 3\pi/2 \right] ). 
$$ 
Putting together the results of the section 6, for every function $ \, w \in \mathcal E^{2,\alpha}_{n,\nu} \, $ small enough we can construct a surface $\tS(w)$ which is close to $\tS$ and whose mean curvature can be expressed as
$$ 
\mathcal H(w) = \mathcal H(0) + \mathcal L \, w + Q( w ) ,
$$
where $ \mathcal H(0) $ is the mean curvature of $\tS$, $ \mathcal L $ is a linear differential operator, which has the form
$$ 
\mathcal L = \left\{ \begin{array}{l} L_{gr} + \frac{\e^{2-\beta}}{\gamma^4} \, L^\gamma_z  \quad \mbox{in} \quad \Omega_{gr} \cup \Omega_{glu}^0 \underset{m=1}{\overset{n} \cup} \Omega_{glu}^m \\[2mm]
L_{cat} + (1 + \frac{\e^{2 -\beta}}{\gamma^2}) \, L_{s,\phi} \quad \mbox{in} \quad \Omega_{cat}^0 \\[2mm] 
\left( \bar L_{cat} + \frac{\e^{-\beta}}{\gamma} \, L_{\sigma,\theta} \right)( \ \ \circ \lambda_m ) \quad \mbox{in} \quad \Omega_{cat}^m \end{array} \right. 
$$
and $Q$ is the nonlinear part of $\mathcal H(w)$ which can be written in the form
$$ 
Q(w) = \dfrac{\e^{1-\beta}}{\gamma^4} \, Q_z^{2,\gamma}(w) + \dfrac{\e^{-\beta}}{\gamma^4} Q_z^{3,\gamma}(w) ,
$$
where the properties of $L^\gamma$, $Q_2^\gamma$ and $ Q_3^\gamma $ are described in the section 6. First, we verify that
$$ 
\| \gamma^2 \, \mathcal H(0) \|_{\mathcal C^{0,\alpha}_\nu(\Omega_\e)} \leq c \, \e^{5/3 - \beta - \nu}, \quad \forall \beta \in (0,1) .
$$
It follows from the fact that away from $0$ and the $n$-th roots of unity, where $\tilde S_n$ is parametrized as a graph of one of the functions $ \pm \, \tilde{\mathcal G}_n $, the mean curvature satisfies
$$ \mathcal H(0) = \left| P_3(\tilde{\mathcal G}_n) \right|, $$
and its norm is bounded by a constant times $ \e^{3 - \nu - \beta} $. On the other hand, in the gluing regions the mean curvature is bounded by a constant times $ \e^{3-\beta} / \gamma^4 $ and in the catenoidal regions by $ \e/\gamma $.

\medskip

Secondly, there exist constants $ c \in \mathbb R $ and $ p \in \mathbb N $, such that
\begin{equation*} \label{estQ}
 \| \gamma^2 \, Q(w) \|_{\mathcal C^{0,\alpha}_\nu ( \Omega_\e ) } \leq c \, \e^{2/3 - p(\alpha + \nu + \beta)} \, \| w \|_{\mathcal C^{2,\alpha}_\nu 
 (\Omega_\e) }, \quad \mbox{for} \quad  \| w \|_{\mathcal C^{2,\alpha}_\nu (\Omega_\e) } \leq c \, \e^{5/3 - \nu - \beta}.
\end{equation*}

The surface $\tS(w)$ is minimal if and only if
$$ 
\mathcal L \, w =  - \mathcal H(0) -  Q(w) .
$$
If $\mathcal L$ is an invertible linear continuous operator then function $w$ should satisfy 
\begin{equation} \label{FixedPointEq} w = - \mathcal L^{-1} \, \left( \mathcal H(0) + Q(w) \right) = \mathcal A(w) \end{equation}
If we show that there exists an open ball $ \mathcal B \subset \mathcal E^{k,\alpha}_{n,\nu} $ such that $ \mathcal A : \mathcal B \longrightarrow \mathcal B $ is a contraction mapping, then by the Banach Fixed Point theorem, there will exist a unique function $w_n$ a solution to \eqref{FixedPointEq} such that  $ \tilde \Sigma_n = \tS(w_n) $ and $ \Sigma_n = \mathcal S_n(w_n) $ are free boundary minimal surfaces in $B^3$.

\subsection{Inverse Linear Operator}

We would like to find a linear operator 
$$ 
\mathcal M \, : \,  \mathcal E^{0,\alpha}_{n,\nu} \longrightarrow \mathcal E_{n,\nu}^{2,\alpha}, \quad \mbox{such that} \quad \mathbb \gamma^2 \, \mathcal L \circ \mathcal M (f) = f .
$$
Take a partition of unity on the unit disk $D^2$~: 
$$ 
\varphi_i \in \mathcal C^{\infty}(D^2), \quad \mbox{such that} \quad {\sum_{i = 0}^n \varphi_i = 1}, \quad \mbox{and} \quad \varphi_i = \delta_{ij} \quad \mbox{in} \quad U_i \supset \ z_i, 
$$
where $z_0 = 0$ and $z_m$, $m=1, \ldots, n$ are the $n$-th roots of unity and $U_i$ are small neighborhoods of $z_i$. Given function $f \in \mathcal E^{0,\alpha}_{n,\nu}$ we can decompose it as
$$ 
f = \sum_{i=0}^n \varphi_i \, f = \sum_{i=0}^n f_i, \quad \mathrm{supp}(f_i) \subset U_i. 
$$

Let us fix the coordinates $ (s,\phi) $ and $ (\sigma, \theta) $ and take $s_{\te} \in \mathbb R_+$ such that 
$$ \te \cosh s_{\te} = 1. $$ 
We can parametrize two copies of the unit punctured disk $ D^2 $ by 
\begin{align*} 
 & z_{+} = r_{+} \, e^{i \phi}, \quad \mbox{where} \quad r_{+} = e^{s -s_{\te}}, \quad s \in (- \infty, s_{\te}) \quad \mbox{and} \\
 & z_{-} = r_{-} \, e^{i \phi}, \quad \mbox{where} \quad r_{-} = e^{- s - s_{\te}}, \quad s \in (- s_{\te}, + \infty) 
\end{align*} 
Remark that we can also parametrize two copies of the punctured half-plane $ \mathbb C_- $ by
\begin{align*} 
 & \zeta_{+} = \rho_{+} \, e^{i \theta}, \quad \mbox{where} \quad \rho_{+} = e^{\sigma}, \quad \sigma \in (- \infty, \infty) \quad \mbox{and} \\
 & \zeta_{-} = \rho_{-} \, e^{i \theta}, \quad \mbox{where} \quad \rho_{-} = e^{-\sigma}, \quad \sigma \in (- \infty, + \infty) 
\end{align*}
and take the conformal mappings $  \lambda_m^\pm \, : \, \mathbb C_- \longrightarrow D^2 $ given by
$$  \lambda_m^\pm(\zeta_\pm) = e^{\frac{2 \pi i m}{n}}\dfrac{ 1 + \zeta_\pm }{ 1-\zeta_\pm}. $$ 

We define the cut-off functions $ \vartheta_0 \in \mathcal C^{\infty}\left( [-s_{\te},s_{\te}] \right) $ and $ \bar \vartheta \in \mathcal C^\infty(\mathbb R) $ such that
\begin{align*} 
 & \vartheta \equiv 1, \quad \mbox{for} \quad s > 1, \quad \vartheta \equiv 0, \quad \mbox{for} \quad s < - 1 \\
 & \bar \vartheta \equiv 1, \quad \mbox{for} \quad \sigma > 1, \quad \bar \vartheta \equiv 0, \quad \mbox{for} \quad \sigma < - 1
\end{align*}
Take $ \, f_0(s,\phi) = f_0(\te \cosh s \, e^{i\phi}) \, $ and $ \, \bar f_*(\sigma, \theta) = f_m \circ \lambda_m(\e \cosh \sigma \, e^{i \theta}). \, $ We can decompose 
\begin{align*} 
 & f_0(s,\phi) = f_0^{+}(s,\phi) + f_0^{-}(s,\phi) = \vartheta_0(s) \, f_0(s,\phi) + (1 - \vartheta_0(s)) \, f_0(s,\phi), \quad \mbox{and} \\[3mm]
 & \bar f_*(\sigma,\theta) = \bar f_*^{+}(\sigma,\theta) + \bar f_*^{-}(\sigma,\theta) = \bar \vartheta(\sigma) \, \bar f_*(\sigma, \theta) + (1 - \bar  
 \vartheta(\sigma)) \bar f_*(\sigma,\theta)
\end{align*} 

\medskip

We can extend the function $ f_0^+ $ by zero to the interval $ (-\infty, s_{\te}) $. It defines a function $ \breve f_0^+ $ on the unit punctured disc, parametrized by the variable $ z_+ $. In the same manner, we can extend the function $ f_0^- $ by zero to to interval $ (s_{\te}, + \infty) $ and that defines a function $ \breve f_0^- $ on the unit punctured disk parametrized by $ z_- $. We have
$$ 
\breve f_0^\pm(r_\pm \, e^{i \phi}) = f_0^\pm(\pm \log r_\pm + s_{\te}, \phi) .
$$
We also define the functions $ \mathring f_*^{\pm} $ on the half-plane $ \mathbb C_- $ by $ \mathring f_*^\pm(\rho_\pm \, e^{ i \theta}) = \bar f_*^\pm (\pm \log \rho_\pm,\theta) \,  $. Finally, we put
$$ 
\breve f_m^\pm(z_\pm) = \mathring f_*^\pm \circ \left( \lambda_m^\pm \right)^{-1} \quad \mbox{and} \quad \breve f^\pm = \sum_{i=0}^n \, \breve f_i^\pm ,
$$
each of the functions $ \breve f^{\pm} $ on one of the two copies of $ \overline D^2 \setminus \{0,z_1, \ldots, z_n \} $.

\medskip

\textbf{First approximate solution:} Using the results of the section 7, we find functions $ \breve w_{gr}^\pm \in \mathcal E_{2,\alpha}^{n,\nu} \oplus \mathfrak D_n \oplus \mathfrak D_{\chi_n} $, solutions to the problems
$$ 
\left\{ \begin{array}{l} \gamma^2(z_\pm) \, \Delta \left( B(z_\pm) \, \breve w_{gr}^\pm \right) = \breve f_{\pm} \quad \mbox{in} \quad D^2_* \\[3mm] 
   \partial_{r_\pm} \breve w_{gr}^\pm = 0 \quad \mbox{in} \quad S^1 \setminus \{z_1, \ldots, z_n \} \end{array} \right. 
$$
We have
\begin{align*} 
 \breve w_{gr}^\pm(z_\pm) = & \breve \psi^\pm(z_\pm) + n \, c_0^* + c_1^* \, \chi_n(z_\pm), \quad \mbox{and} \\[3mm] 
 & \| \breve w_{gr}^\pm \|_{\mathcal C^{2,\alpha}_\nu(D^2_*) \oplus \mathfrak D_n \oplus \mathfrak D_{\chi_n}} \leq C \, \| f \|_{\mathcal C^{2, 
 \alpha}_\nu(D^2_*)}
\end{align*} 
We put
\begin{multline*} 
 \psi^{+}(s,\phi) = \breve \psi^{+}(e^{ s-s_{\te} } \, e^{i \phi}) \in (\te \cosh s)^{\nu} \mathcal C^{2,\alpha}( (-\infty,s_{\te}) \times 
 S^1), \\[3mm] 
 \psi^{-}(s,\phi) = \breve \psi^{-}(e^{ -s - s_{\te} } \, e^{i \phi}) \in (\e \cosh s)^{\nu} \mathcal C^{2,\alpha}( (- s_{\te}, + \infty) 
 \times S^1) ,
\end{multline*}
\newline and $ \bar \psi^\pm = \psi^\pm \circ \lambda^\pm_m \in (\e \cosh s)^{\nu} \mathcal C^{2,\alpha}(\Cyl) $. 

\medskip

\medskip

\textbf{First estimate of the error:} We would like now to analyse the behaviour of the function
$$ h := \gamma^2 \mathcal L \, (\psi^+ + \psi^-) - f $$
Remark, that the change of variables $ z = \te \, \cosh s \, e^{i \phi} $ transforms
$$ \Delta_z \leadsto  \dfrac{1}{ {\te}^2 \cosh^2 s} \left(  \coth^2 s \, \partial_s^2 + \coth s \, \partial_s (1 - \coth^2 s) + \partial_\phi^2 \right) $$
On the other hand, the change of variables $ z_\pm = e^{\pm s - s_\e} \, e^{i\phi} $ transforms
$$ \Delta_{\pm z} \leadsto e^{ \mp \, 2(s-s_\e) } \, \left( \partial_s^2 + \partial_\phi^2 \right) $$
In this section we will denote as $c$ any positive constant which does not depend on $\e$. Moreover, $ \left| e^{\pm s - s_\e} - \e \cosh s \right| \leq \frac{c \, \e}{\cosh s} $. Using once the again the partition of unity, we decompose 
$$ h = \sum_{i=0}^n \phi_i \, h = \sum_{i=0}^n h_i, \quad \bar h_* = h_m \circ \lambda_m. $$
Regarding $ h_0 $ as a function in variables $(s,\phi)$ we can extend it by $0$ to $\mathbb R \times S^1$. Similarly, regarding $\bar h_*$ as a function of $(\sigma,\theta)$ we can extend it by $0$ to $ \Cyl $. Fix $ \delta \in (-1,0) $. Then, there exists a universal constant $C$, such that
\begin{multline*} 
 \| (\cosh s)^{-\delta} h_0 \|_{\mathcal C^{0,\alpha}(\mathbb R \times S^1)} \leq C \, \| f \|_{\mathcal C^{0,\alpha}_\nu(\Omega_\e)}, \\[3mm] 
 \| (\cosh \sigma)^{-\delta} \bar h_* \|_{\mathcal C^{0,\alpha}(\Cyl)} \leq C \, \| f \|_{\mathcal C^{0,\alpha}_\nu(\Omega_\e)}
\end{multline*} 

\medskip

\textbf{Help of the linear analysis around the catenoids:} Using the results of the sections 7 and 8, we can find functions $ w_{cat}^0 $ and $ \bar w_{cat}^* $, such that
$$ \frac{\gamma^2}{ 2 \, \te^2 \cosh^2 s} \, L_{cat} \, w_{cat}^0 = h_0 \quad \mbox{and} \quad \frac{\bar \gamma^2}{\e^2 \cosh^2 \sigma} \, \bar L_{cat} \, \bar w_{cat}^* = \bar h_*, \quad \partial_\theta \, \bar w_{cat}^* = 0, $$
where
$$ L_{cat} =  \partial_s^2 + \partial_\phi^2 + \dfrac{2}{\cosh^2 s} \quad \mbox{and} \quad \bar L_{cat} = \partial_\sigma^2 + \partial_\theta^2 + \dfrac{2}{\cosh^2 \sigma} $$
and $ \bar \gamma = \gamma \circ \lambda_m $. We can write $ w_{cat}^0 = v_{cat}^0 + d_0^*$, and $ \bar w_{cat}^* = \bar v_{cat}^* + d_1^*, $ where
\begin{multline*} 
 \| (\cosh s)^{-\delta} v_{cat}^0 \|_{\mathcal C^{2,\alpha}(\mathbb R \times S^1)} + |d_0^*| \leq C \, \| (\cosh s)^{-\delta} h^\e_0 \|_{\mathcal C^{0,
 \alpha}(\mathbb R \times S^1)} \quad \mbox{and} \\[3mm]
 \| (\cosh \sigma)^{-\delta} \bar v_{cat}^* \|_{\mathcal C^{2,\alpha}(\Cyl)} + |d_1^*| \leq \| (\cosh \sigma)^{-\delta} \bar h_* \|_{\mathcal C^{0,\alpha}(\Cyl)}
\end{multline*} 
\newline We also denote $ w_{cat}^m = \bar w_{cat}^* \circ \lambda_m $ and $ v_{cat}^m =\bar v_{cat}^* $. We have 
$$ \left. \partial_r w_{cat}^m \right|_{r=1} = \left. \partial_r v_{cat}^m \right|_{r=1} = 0. $$

\medskip

\textbf{Cut-off functions:} Let, as before, $ \eta_{\te}^0 \in \mathcal C^{\infty}(D^2) $ denote a cut-off function, such that
$$ \eta^0_{\te}(z) = \eta_{\te}^0(|z|), \quad \eta_{\te}^0(z) \equiv 1 \quad \mbox{for} \quad |z| < 1/2 \, \te^{1/2} \quad \mbox{and} \quad \eta_{\te}^0(z) \equiv 0 \quad \mbox{for} \quad |z| > 2 \, \te^{1/2}  .$$
and $ \bar \eta_\e \in \mathcal C^{\infty}(\mathbb C_-) $ the cut-off function in $\mathbb C_-$, such that
$$ \bar \eta_\e(\zeta) = \bar \eta_\e(|\zeta|), \quad \bar \eta_{\e}(\zeta) \equiv 1 \quad \mbox{for} \quad |\zeta| < 1/2 \, \e^{2/3} \quad \mbox{and} \quad \bar \eta_{\e}(\zeta) \equiv 0 \quad \mbox{for} \quad |\zeta| > 2 \, \e^{2/3}. $$
We put $ \eta_\e^m : = \bar \eta_\e \circ \lambda_m^{-1}. $

\medskip

Furthermore, we introduce the cut-functions $\varkappa_{\te}^0 \in \mathcal C^{\infty}(D^2) $ and $ \bar \varkappa_\e \in \mathcal C^{\infty}(\mathbb C_-) $, such that
\begin{multline*}
 \varkappa^0_{\te}(z) = \varkappa_{\te}^0(|z|), \quad \varkappa_{\te}^0(z) \equiv 1 \quad \mbox{for} \quad |z| < 2 \, \te^{1/2} \quad \mbox{and} \quad 
 \varkappa_{\te}^0(z) \equiv 0 \quad \mbox{for} \quad |z| > 3 \, \te^{1/2} \\
 \bar \varkappa_\e(\zeta) = \bar \varkappa_\e(|\zeta|), \quad \bar \varkappa_{\e}(\zeta) \equiv 1 \quad \mbox{for} \quad |\zeta| < 2 \, \e^{2/3} \quad 
 \mbox{and} \quad \bar \varkappa_{\e}(\zeta) \equiv 0 \quad \mbox{for} \quad |\zeta| > 3 \, \e^{2/3}
\end{multline*}
and $ \varkappa_\e^m := \bar \varkappa_\e \circ \lambda_m^{-1}. $

\medskip

Let $I_{\te}^+ = [s^g_{\te}, S^g_{\te}] \quad \mbox{and} \quad I_{\te}^- = [- S^g_{\te}, - s^g_{\te}] \, $
be the two subintervals of $ [-s_{\te},s_{\te}] $ where we glue together the graph of the functions $ \pm \, \tilde{\mathcal G}_n $ with the catenoidal neck. 

\medskip

Similarly, let $ \, J^+_\e = [\sigma^g_\e, \Sigma^g_\e] \quad \mbox{and} \quad J_\e^- = [-\Sigma_\e^g, - \sigma_\e^g]  \, $ be the two intervals where we glue together the graph of the functions $ \pm \, \bar{ \mathcal G }_n$ with the half-catenoidal necks. We put
$$ R_{\te} = e^{ - s_{\te}^g - s_{\te} }, \quad  r_{\te}  = e^{ - S_{\te}^g - s_{\te} } \quad \mbox{and} \quad \mathcal P_\e = e^{ - \sigma_\e^g }, \quad \rho_\e = e^{-\Sigma_\e^g}. $$
We denote $ \xi_{\te}^0 \in \mathcal C^{\infty}(D^2) $ a cut-off function such that
$$ 
\xi_{\te}^0(z) = \xi_{\te}^0(|z|), \quad \xi_{\te}^0(z) \equiv 1 \quad \mbox{for} \quad |z| > R_{\te}, \quad \mbox{and} \quad \xi_{\te}^0(z) \equiv 0 \quad |z| < r_{\te}  .
$$
In the same manner we define the cut-off function $ \bar \xi_\e \in \mathcal C^{\infty}(\mathbb C) $, such that
$$ \bar \xi_{\e}(\zeta) = \bar \xi_{\e}(|\zeta|), \quad \bar \xi_{\e}(\zeta) \equiv 1 \quad \mbox{for} \quad |\zeta| > \mathcal P_{\e}, \quad \mbox{and} \quad \bar \xi_{\e}(\zeta) \equiv 0 \quad |\zeta| < \rho_{\te}, $$
and we put
$$ \xi_\e^m : = \bar \xi_\e \circ \lambda_m^{-1}. $$
Finally, we define the cut-off function $ \xi_\e \in \mathcal C^{\infty}(D^2) $ : $ \xi_\e : = \xi_{\te}^0 \, \underset{m=1}{\overset{n} \Pi} \xi_\e^m, $
and also the functions 
$$ \xi_\e^\pm \in \mathcal C^\infty( [ -s_{\te}, s_{\te} ] \times S^1 ) : \quad \xi_\e^\pm(s,\phi) = \xi_\e( e^{ \pm s - s_{\te} } \, e^{i \phi} ). $$
 
Notice, that $ \xi_\e^+ $ is a function which is equal to $1$ on the part of the surface, which is parametrized as a graph of the function $ - \tilde{ \mathcal G}_n $, the parts, where we glue this graph with upper parts of the catenoidal neck and the half-catenoidal bridges and also on the neck and the bridges them selves. It is identically equal to zero on the part of the surface, which is parametrized as a graph of the function $ \tilde{ \mathcal G}_n $. In the same manner, one can easily deduce the properties of $\xi_\e^-$.

\medskip

\textbf{Regular terms:} Consider the function
$$ w_{reg} := \xi_\e^+ \, \psi_{gr}^+ \, + \, \xi_\e^- \ \psi_{gr}^- + \eta_{\te}^0 \, v_{cat}^0 + \sum_{m=1}^n \, \eta_\e^m \, v_{cat}^m. $$

\medskip

\textbf{Deficiency terms:} We define the functions 
$$ u_0(s): = 1 - s \, \tanh s, \quad \bar u(\sigma): = 1 - \sigma \, \tanh \sigma, \quad u_m: =  \bar u \circ \lambda_m^{-1}, $$
and take $\Gamma_n$ and $\tilde \Gamma_n$ the functions lying in the kernel of the operator $ L_{gr} $ and defined section 4. We also denote
$$ \Gamma_n^\pm(s,\phi) = \Gamma_n( e^{ \pm s - s_{\te} } \, e^{i \phi} ), \quad \tilde \Gamma_n^\pm (s,\phi) = \tilde \Gamma_n( e^{ \pm  s - s_{\te} } \, e^{i \phi} ). $$

Consider the function
\begin{align*}
 \kappa_\e(s,\phi) & : = \varkappa_{\te}^0 \, ( a_0 \, u_0 + d_0^* ) +  \sum_{m=1}^n \varkappa_\e^m \, ( a_1 \, u_m + d_1^*) \\
 & + ( 1 -\varkappa_{\te}^0 - \sum_{m=1}^n \varkappa_\e^m ) \Big[ c_0^* \, n + c_1^* \, \chi_n + b_0 \, \left( \xi_\e^{+} \, \tilde \Gamma_n^{+} + \xi_\e^{-} 
 \, {\tilde \Gamma}_n^{-} \right) \\
 & + b_1 \, \left( \xi_\e^+ \, \Gamma_n^{+} + \xi_\e^- \, \Gamma_n^{-} \right) \Big]
\end{align*}
where the constants $a_0$, $a_1$, $b_0$ and $b_1$ are chosen in such a way that the norm of $ \kappa $ and its derivatives would be small in $\Omega_\kappa^0 \underset{m=1}{\overset{n} \cup } \Omega_\kappa^m $, where
$$ \Omega_\kappa^0 : = \mathrm{supp} (\varkappa_{\te}^0) \cap \mathrm{supp}(1 - \varkappa_{\te}^0), \quad \mbox{and} \quad \Omega_\kappa^m := \mathrm{supp} (\varkappa_{\e}^m) \cap \mathrm{supp}(1 - \varkappa_{\e}^m). $$

\medskip

More precisely, in $ \Omega_\kappa^0 $ we have
$$ \xi_\e^{+} \, \tilde \Gamma_n^{+} + \xi_\e^{-} \, \tilde \Gamma_n^{-} =
\left\{ \begin{array}{l} -2n + 2n \, s_{\te} - 2n \, s + \mathcal O(\e), \quad \mbox{when} \quad s > 0 \\[3mm] 
-2n + 2n \, s_{\te} + 2n \ s + \mathcal O(\e), \quad \mbox{when} \quad s < 0 \end{array} \right. $$
On the other hand,
$$ u_0(s) = \left\{ \begin{array}{l} 1 - s + \mathcal O(\te), \quad \mbox{for} \quad s > 0 \\[3mm]
1 + s + \mathcal O(\te) \quad \mbox{for} \quad s < 0 \end{array} \right. ,
$$
and
$$ \xi_\e^{+} \, \Gamma_n^{+} + \xi_\e^{-} \, \Gamma_n^{-} = - \dfrac{n}{2} + \mathcal O(\te), \quad \chi_n(s,\phi) \equiv 0, \quad \varkappa_\e^m \equiv 0 $$
\newline This gives us the first equation on $ a_0, a_1, b_0, b_1 $:
\begin{equation} \label{GluKer1} 
 a_0 = 2n \, b_0 , \quad c_0^* \, n - \dfrac{b_1 \, n}{2} - 2n \, b_0 + 2n \, b_0 \, s_{\te} = 2n \, b_0 + d_0^* .
\end{equation}

Similarly, in $ \Omega_\kappa^m $ we have for all $\beta \in (0,1)$
$$ \xi_\e^{+} \, \Gamma_n^{+} + \xi_\e^{-} \, \Gamma_n^{-}  = \left\{ \begin{array}{l} - \dfrac{n}{2} + c(n) - \sigma + \mathcal O(\e^{2/3 - \beta}), \quad \mbox{when} \quad \sigma > 0 \\[3mm]
- \dfrac{n}{2} + c(n) + \sigma + \mathcal O(\e^{2/3 - \beta}), \quad \mbox{when} \quad \sigma < 0
\end{array} \right. $$

$$ \bar u(\sigma) = \left\{ \begin{array}{l} 1 - \sigma + \mathcal O(\e^{2/3}), \quad \mbox{for} \quad \sigma > 0 \\[3mm] 
1 + \sigma + \mathcal O(\e^{2/3}) \quad \mbox{for} \quad \sigma < 0 \end{array} \right. ,
$$
and
$$ 
\tilde \Gamma_n(s,\phi) = - n + \mathcal O(\e^{2/3}), \quad \chi_n(s,\phi) \equiv 1, \quad \varkappa_{\te}^0 \equiv 0 .
$$
\newline This gives us the second equation
\begin{equation} \label{GluKer2}
 a_1 = b_1, \quad c_0^* \, n + c_1^* + b_1( c(n) - \dfrac{n}{2}) - b_0 \, n = b_1 + d_0^* .
\end{equation}
Then, the system \eqref{GluKer1} and \eqref{GluKer2} has a unique solution.

\medskip

\textbf{Second approximate solution}: As an approximate solution, we take the function
$$ w_{app} : = w_{reg} + \kappa. $$
We would like to estimate the norm in $\mathcal E^{2,\alpha}_{n,\nu}$ of the function 
$$ \gamma^2 \, \mathcal L \, w_{app} - f. $$ 
We do it separately in the regions $ \Omega_{cat}^0 $, $ \Omega_{glu}^0 $, $ \Omega_\kappa^0 $, $ \Omega_{gr} $, $ \Omega_\kappa^m $, $ \Omega_{glu}^m $ and $\Omega_{cat}^m $.

\medskip

\begin{enumerate}
 
\item In $ \Omega_{gr}^+ $, where the surface $\tS$ is parametrized as a graph of the function $- \tilde{\mathcal G}_n$ we have
$$ 
\eta_{\te}^0 \equiv \eta_\e^m \equiv 0, \quad \mbox{and} \quad \xi_\e^+ \equiv 1, \quad \xi_\e^- \equiv 0 .
$$
We see that 
\begin{multline*} 
 f = f^+, \quad w_{app} = \psi_{gr}^{+} + \kappa, \quad \mbox{and} \quad \gamma^2 \, \mathcal L = \gamma^2 \, L_{gr} + \dfrac{\e^{2-\beta}}{\gamma^2} \, 
 L^{\gamma} \\ 
 \kappa = c_0^* \, n + c_1^* \, \chi_n + b_0 \, \tilde \Gamma_n^{+} + b_1 \, \Gamma_n^{+} \quad \mbox{in} \quad \Omega_\kappa^c \quad \mbox{and} \quad \kappa 
 = \mathcal O(\te) \quad \mbox{in} \quad \Omega_\kappa \cap \Omega_{gr}^{+}.
\end{multline*}
Then,
\begin{equation*} \label{EstLinDif1} 
 \| \gamma^2 \, \mathcal L \, w_{app} - f \|_{\mathcal C^{0,\alpha}_\nu(\Omega_{gr}^+)} \leq c \, \e^{2/3 - \beta - \nu} \, \|f \|_{\mathcal C^{0,\alpha}_\nu(\Omega_\e)}, \quad \forall \beta \in (0,1).
\end{equation*}
Naturally, the same estimate is true in $ \Omega_{gr}^{-} $, where the surface is parametrized as a graph of $ \tilde{\mathcal G}_n $ and where
$$ \eta_{\te}^0 \equiv \eta_\e^m \equiv 0, \quad \mbox{and} \quad \xi_\e^- \equiv 1, \quad \xi_\e^+ \equiv 0, $$
and $ f = f^- $, $ w_{app} = \psi_{gr}^{-} + \kappa $.

\medskip

\item In the region $ \Omega_{glu}^{0,+} $, where we glue together the upper part of the catenoidal neck with the graph of the function $ - \tilde{\mathcal G}_n $ we have
$$ \eta_\e^m \equiv 0, \quad \xi_\e^+ \equiv 1 \quad \mbox{and} $$
\begin{multline*} 
 f = f^+, \quad w_{app} = \psi_{gr}^{+} + \xi_\e^- \, \psi_{gr}^{-} - \eta_{\te}^0 \, v_{cat}^0 + \kappa, \\
 \kappa = a_0 \, u_0 + d_0^* \quad \mbox{and} \quad \gamma^2 \, \mathcal L = \gamma^2 \, L_{gr} + \e^{1-\beta} \, L^{\gamma_0}. 
\end{multline*}  
We obtain for $\e$ small enough
\begin{multline*} 
 \| \gamma^2 \, \mathcal L \, w_{app} - f \|_{\mathcal C^{0,\alpha}_\nu(\Omega_{glu}^{0,+})}\\[3mm]  
 \leq c \, \left( \| \psi^-_{gr} \|_{\mathcal C^{2,\alpha}_\nu(\Omega_{glu}^{+,0})} + \| v^0_{cat} \|_{\mathcal C^{2,\alpha}_\nu(\Omega_{glu}^{+,0})} + \| \kappa \|_{\mathcal C^{2,\alpha}_\nu(\Omega_{glu}^{+,0})} \right) \leq c \, \e^{1/2|\delta| - \nu - \beta} \| f \|_{\mathcal C^{0,\alpha}_\nu(\Omega_\e)}
\end{multline*}
\newline By symmetry the same estimate is true in the region $\Omega_{glu}^{0,-}$, where we glue the lower part of the catenoidal neck with the graph of the function $\tilde {\mathcal G}_n$ and where
$$ \eta_\e^m \equiv 0, \quad \xi_\e^- \equiv 1 \quad \mbox{and} $$
$$ f = f^-, \quad w_{app} = \psi_{gr}^{-} + \xi_\e^+ \, \psi_{gr}^{+} - \eta_{\te}^0 \, v_{cat}^0  + \kappa, \quad \kappa = a_0 \, u_0 + d_0^* $$

\medskip

\item In the regions $ \Omega_{glu}^{m,+} $ where we glue together the upper part of the half-catenoidal bridges with the graph of the function $ - \tilde{ \mathcal G}_n $ we have

$$ 
\eta_{\te}^0 \equiv 0, \quad \xi_\e^+ \equiv 1 \quad \mbox{and} 
$$
\begin{multline*}
 f = f^+, \quad w_{app} = \psi_{gr}^{+} + \xi_\e^- \, \psi_{gr}^{-} - \eta_{\te}^m \, v_{cat}^m + \kappa \\
 \kappa = a_1 \, u_m + d_1^* \quad \mbox{and} \quad \gamma^2 \, \mathcal L = \gamma^2 \, L_{gr} + \e^{2/3 - \beta } \, L^{\gamma_m} .
\end{multline*}
We obtain 
\begin{multline*} 
 \| \gamma^2 \, \mathcal L \, w_{app} - f \|_{\mathcal C^{0,\alpha}_\nu(\Omega_{glu}^{m,+})} \\[3mm]  
 \leq c \, \left( \| \psi^-_{gr} \|_{\mathcal C^{2,\alpha}_\nu(\Omega_{glu}^{+,m})} + \| v^m_{cat} \|_{\mathcal C^{2,\alpha}_\nu(\Omega_{glu}^{+,m})} + \| \kappa \|_{\mathcal C^{2,\alpha}_\nu(\Omega_{glu}^{+,m})} \right) \leq c \, \e^{2/3|\delta| - \nu - \beta} \, \| f \|_{\mathcal C^{0,\alpha}_\nu(\Omega_\e)}
\end{multline*}
\newline Once again, by symmetry the same estimate is true in $\Omega_{glu}^{m,-}$.

\medskip

\item In the catenoidal region $ \Omega_{cat}^0 $, where the surface $ \tS $ coincides with the catenoidal neck, we have
$$ \eta_\e^m \equiv 0, \quad \xi_\e^+ \equiv \xi_\e^- \equiv 1 . $$
\begin{multline*} 
 f = f^{+} + f^{-}, \quad w_{app} = \psi_{gr}^{-} + \psi_{gr}^{+} - w_{cat}^0, \quad 
 \gamma^2 \, \mathcal L = \frac{\gamma^2}{ 2 \, \te^2 \cosh^2 s} \, L_{cat} + \e^{1-\beta} \, L
\end{multline*} 
We obtain
\begin{equation*}
 \| \gamma^2 \, \mathcal L \, w_{app} - f  \|_{\mathcal C^{0,\alpha}_\nu(\Omega_{cat}^0)} \leq c \, \e^{1-\nu-\beta} \, \| f \|_{\mathcal C^{0,
 \alpha}_\nu(\Omega_\e)}.
\end{equation*}

\medskip

\item Finally, in the region $ \Omega_{cat}^{m} $, where the surface $ \tS $ coincides with the $m$-th half-catenoidal neck, we have
$$ 
\eta_\e^0 \equiv 0, \quad \xi_\e^+ \equiv \xi_\e^- \equiv 1 \quad \mbox{and} 
$$
$$ 
f : = f^{+} + f^{-}, \quad w_{app} = \psi_{gr}^{-} + \psi_{gr}^{+} - w_{cat}^m, \quad \gamma^2 \, \mathcal L ( \ \ \circ \lambda_m ) = \frac{\gamma^2}{ \e^2 \cosh^2 \sigma} \, \bar L_{cat} + \e^{2/3 - \beta} \, L  $$
We obtain
\begin{equation*}
 \| \gamma^2 \, \mathcal L \, w_{app} - f \|_{\mathcal C^{0,\alpha}_\nu(\Omega_{cat}^{m})} \leq c \, \e^{2/3-\nu - \beta} \, \| f \|_{\mathcal C^{0,\alpha}_\nu(\Omega_\e)}.
\end{equation*}

\end{enumerate}

\medskip

Now let us introduce the operator $ \, \mathcal M_{app} \, : \, \mathcal C^{0,\alpha}_\nu(\Omega_\e) \longrightarrow C^{2,\alpha}_\nu(\Omega_\e), \, $ such that
$$ \mathcal M_{app} (f) = w_{app} \quad \mbox{and} \quad \| \mathcal M_{app} \| \leq c \, \e^{-\nu - \beta}, \quad \forall \beta \in (0,1). $$
On the other hand, if we take $ \delta $ sufficiently close to $-1$ and $ \nu $ and $ \beta $ small enough, then the operator
$$ \mathcal R_{app} : = \gamma^2 \, \mathcal L \circ \mathcal M_{app} - \mathrm{Id} \, : \, \mathcal C^{0,\alpha}_\nu(\Omega_\e) \longrightarrow \mathcal C^{0,\alpha}_\nu(\Omega_\e), $$
satisfies
$$ \| R_{app} \| \ll 1. $$
So, $ \, \mathrm{Id} + \mathcal R_{app} \, $ is an invertible operator. We denote
$$ \mathcal M_{exact} : = \mathcal M_{app} \circ \left( \mathrm{Id} + \mathcal R_{app} \right)^{-1}, \quad \mbox{and} \quad \gamma^2 \, \mathcal L \circ \mathcal M_{exact} = \mathrm{Id}. $$

\subsection{Conclusion}

Using the results of the previous subsection, we conclude that there exist constants $ c \in \mathbb R_+ $ and $ p, q \in \mathbb N $, such that
$$ 
\| \mathcal M_{exact} \left( \gamma^2 \, \mathcal H(0) \right) \|_{\mathcal E^{2,\alpha}_{n,\nu}} \leq \, \e^{5/3 - q(\nu + \beta)} ,
$$
and
$$ 
\| \mathcal M_{exact} \left( \gamma^2 \, Q(w) \right) \|_{\mathcal E^{2,\alpha}_{n,\nu}} \leq \, \e^{1/3 - p(\alpha + \nu + \beta)} \, \| w \|_{\mathcal E^{2,\alpha}_{n,\nu}} .
$$
\newline This yields that for $ \alpha, \nu, \beta $ and $\e$ small enough
$$ 
\| \mathcal M_{exact} \left( \gamma^2 \, Q(w_1) \right) - \mathcal M_{exact} \left( \gamma^2 \, Q(w_2) \right) \|_{\mathcal E^{2,\alpha}_{n,\nu}} \leq \dfrac{1}{2} \, \| w_1 - w_2 \|_{\mathcal E^{2,\alpha}_{n,\nu}} ,
$$
\newline and $ \mathcal M_{exact} \left( \gamma^2 \, Q(\cdot) \right) $ is a contracting mapping in the ball
$$ 
\mathcal B_\e : = \left\{ w \in \mathcal E^{2,\alpha}_{n,\nu} \, : \, \| w \|_{\mathcal E^{2,\alpha}_{n,\nu}} \leq  \e^{5/3 - q(\nu + \beta)} \right\} .
$$
And the result of this article follows from the Banach fixed point theorem.

\end{document}